\documentclass[preprint,11pt]{elsarticle}

\usepackage{amsfonts, amsmath, amscd}
\usepackage[psamsfonts]{amssymb}
\usepackage{amssymb}
\usepackage{pb-diagram}
\usepackage[all,cmtip]{xy}
\usepackage[usenames]{color}

\def\cl#1{\overline{#1}}

\def\dag-{$\kappa$-pseudo open }
\def\Int{\operatorname{Int}}

\newtheorem{theorem}{Theorem}[section]
\newtheorem{lemma}[theorem]{Lemma}
\newtheorem{corollary}[theorem]{Corollary}

\newtheorem{remark}[theorem]{Remark}
\newtheorem{proposition}[theorem]{Proposition}
\newtheorem{definition}[theorem]{Definition}
\newtheorem{example}[theorem]{Example}
\newtheorem{fact}[theorem]{Fact}

\newtheorem{problem}[theorem]{Problem}

\newproof{proof}{Proof}

\numberwithin{equation}{section}
\numberwithin{theorem}{section}
\newcommand{\w}{\omega}

\newcommand{\NN}{\mathbb{N}}

\newcommand{\IR}{\mathbb{R}}

\newcommand{\KK}{\mathcal{K}}

\newcommand{\AAA}{\mathcal{A}}

\newcommand{\pr}{\mathrm{pr}}
\newcommand{\Ra}{\Rightarrow}

\newcommand{\LRa}{\Leftrightarrow}

\newcommand{\CC}{C_k}

\newcommand{\SM}{{\setminus}}

\def\cl#1{\overline{#1}}

\begin{document}

\begin{frontmatter}

\title{New classes of compact-type spaces}
%\title{$\kappa$-sequential and (weakly) open-compact attainable spaces}

\author{Saak Gabriyelyan}%\tnotetext[label1]{The first named author was partially supported by Israel Science Foundation grant 1/12.}
%\ead{saak@math.bgu.ac.il}
\address{Department of Mathematics, Ben-Gurion University of the Negev, Beer-Sheva, P.O. 653, Israel}

\author{Evgenii Reznichenko}%\tnotetext[label1]{The second named author was partially supported by Israel Science Foundation grant 1/12.}
%\ead{erezn@inbox.ru}
\address{Department of Mathematics, Lomonosov Mosow State University, Moscow, Russia}

\begin{abstract}
Being motivated by the notions of $\kappa$-Fr\'{e}chet--Urysohn spaces and $k'$-spaces introduced by Arhangel'skii, the notion of sequential spaces and the study of Ascoli spaces, we introduce three new classes of compact-type spaces. They are defined by the possibility to attain each or some of boundary points $x$ of an open set $U$ by a sequence in $U$ converging to $x$ or by a relatively compact subset $A\subseteq U$ such that $x\in \overline{A}$. Relationships of the introduced classes with the classical classes (as, for example, the classes of $\kappa$-Fr\'{e}chet--Urysohn spaces, (sequentially) Ascoli spaces, $k_\IR$-spaces, $s_\IR$-spaces etc.)  are given. We characterize these new classes of spaces and study them with respect to taking products,  subspaces and quotients. In particular, we give new characterizations of $\kappa$-Fr\'{e}chet--Urysohn spaces and show that each feathered topological group is $\kappa$-Fr\'{e}chet--Urysohn. We describe locally compact abelian groups which endowed with the Bohr topology  belong to one of the aforementioned classes. Numerous examples are given.
\end{abstract}

\begin{keyword}
$\kappa$-Fr\'{e}chet--Urysohn space \sep open-compact attainable  space  \sep weakly open-compact attainable  space \sep $\kappa$-sequential  space  \sep Ascoli space \sep $\kappa$-pseudo open map \sep weakly $\kappa$-pseudo open map

\MSC[2020] 22B05 \sep 54A05 \sep  54B05 \sep   54C10 \sep 54D70

\end{keyword}

\end{frontmatter}

%%%%%%%%%%%%%%%%%%
%%%%%%%%%%%%%%%%%%
%%%%%%%%%%%%%%%%%%
%%%%%%%%%%%%%%%%%%
%%%%%%%%%%%%%%%%%%

\section{Introduction}

%%%%%%%%%%%%%%%%%%
%%%%%%%%%%%%%%%%%%
%%%%%%%%%%%%%%%%%%
%%%%%%%%%%%%%%%%%%
%%%%%%%%%%%%%%%%%%

All topological spaces are assumed to be Tychonoff. We denote by $C_p(X)$ and $\CC(X)$ the space $C(X)$ of all continuous real-valued functions on a space $X$ endowed with the pointwise topology or the compact-open topology, respectively. Let us recall some of the most important classes of topological spaces which are widely studied in general topology and functional analysis.

\begin{definition} \label{def:compact-type} {\em A space $X$ is called
\begin{enumerate}
\item[$\bullet$] {\em Fr\'{e}chet--Urysohn} if for each non-closed subset $A\subseteq X$ and every point $a\in \overline{A}\SM A$, there is a sequence $\{a_n\}_{n\in\w}\subseteq A$ converging to $a$;
\item[$\bullet$] a {\em $k'$-space} if for each subset $A\subseteq X$ and every point $a\in \overline{A}$, there is a compact set $K\subseteq X$ such that $a\in \overline{K\cap A}$;
\item[$\bullet$] {\em sequential} if for each non-closed subset $A\subseteq X$ there is a sequence $\{a_n\}_{n\in\NN}\subseteq A$ converging to some point $a\in \bar A\setminus A$;
\item[$\bullet$] an {\em $s_\IR$-space} if every sequentially continuous function $f:X\to\IR$ is continuous ($f$ is {\em sequentially continuous} if  the restriction of $f$ onto each convergent sequence is continuous);
\item[$\bullet$] a {\em $k$-space} if for each non-closed subset $A\subseteq X$ there is a compact subset $K\subseteq X$ such that $A\cap K$ is not closed in $K$;
\item[$\bullet$] a {\em $k_\IR$-space} if every $k$-continuous function $f:X\to\IR$ is continuous ($f$ is {\em $k$-continuous} if the restriction of $f$ onto any compact subset $K\subseteq X$ is continuous);
\item[$\bullet$] {\em $\kappa$-Fr\'{e}chet--Urysohn} if for every open subset $U$ of $X$ and every $x\in \overline{U}$, there exists a sequence $\{x_n\}_{n\in\w} \subseteq U$ converging to $x$;
\item[$\bullet$] an {\em Ascoli space} if every compact subset of $\CC(X)$ is evenly continuous;
\item[$\bullet$] a {\em sequentially Ascoli space} if every convergent sequence in $\CC(X)$ is equicontinuous.
\end{enumerate} }
\end{definition}
The classes of Fr\'{e}chet--Urysohn spaces, sequential spaces,  $k$-spaces and $k_\IR$-spaces are classical, see for example the text of Engelking \cite{Eng} or  Michael's article \cite{Mi73}. $k'$-spaces were introduced by Arhangel'skii \cite{Arhangel63,Arhan63}. The Ascoli theorem states that if $X$ is a $k$-space, then every compact subset of $\CC(X)$ is evenly continuous, see Theorem 3.4.20 in \cite{Eng}. In \cite{Noble}, Noble proved that every $k_\IR$-space satisfies the conclusion of the Ascoli theorem. Being motivated by these results, the class of Ascoli spaces was separated and studied by Banakh and Gabriyelyan \cite{BG}. Sequentially Ascoli spaces were defined in \cite{Gabr:weak-bar-L(X)}. These properties were intensively studied for function spaces, free locally convex spaces, strict $(LF)$-spaces and their strong duals, Banach and Fr\'{e}chet spaces in the weak topology etc., see for example \cite{Arhangel,Gabr,Jar,mcoy,S-W} and references therein. $s_\IR$-spaces were introduced by Noble in \cite{Nob}  and thoroughly studied very recently by the authors in \cite{GR-kR-sR,GR-prod}.

The property of being a $\kappa$-Fr\'{e}chet--Urysohn space was introduced by Arhangel'skii and studied by Liu and Ludwig \cite{LiL}. The following interesting result was obtained by Sakai \cite{Sak2}.
\begin{theorem}[\cite{Sak2}] \label{t:Cp-kappa-FU}
The space $C_p(X)$ is $\kappa$-Fr\'{e}chet--Urysohn if, and only if, $X$ has the property $(\kappa)$.
\end{theorem}
The $\kappa$-Fr\'{e}chet--Urysohn  spaces $\CC(X)$ were characterized also by Sakai in \cite{Sakai-3}. In \cite{GGKZ}, it is proved that if $C_p(X)$ is Ascoli, then it is $\kappa$-Fr\'{e}chet--Urysohn. The converse assertion is proved in \cite{Gabr-B1}. Moreover, in \cite{Gabr-seq-Ascoli} the last two assertions were generalized to sequentially Ascoli spaces, and therefore we have the following theorem.

\begin{theorem} \label{t:Cp-Ascoli}
The space $C_p(X)$ is Ascoli if, and only if, it is sequentially Ascoli if, and only if, it is $\kappa$-Fr\'{e}chet--Urysohn. %$X$ has the property $(\kappa)$.
\end{theorem}

The most important part in the proof of the converse assertion is the fact that any  $\kappa$-Fr\'{e}chet--Urysohn space is Ascoli, see Theorem 2.5 of \cite{Gabr-B1}. This fact immediately follows from a more general result:  A space $X$ is Ascoli if it satisfies the next condition: {\em $(\star)$ for every open subset $W$ of $X$ and every $z\in \overline{W}$ there exists $L\subseteq W$ such that $\overline{L}$ is compact and $z\in \overline{L}$.}
The condition $(\star)$ is a natural  analogue of the Arhangel'skii notion of $\kappa$-Fr\'{e}chet--Urysohn spaces in which a ``convergent sequence'' is replaced by a ``relatively compact subset''. In other words, $\kappa$-Fr\'{e}chet--Urysohness means that for each {\em open} set $W$, any  point $x\in \overline{W}\SM W$  can be {\em attained} by a {\em convergent sequence} lied in $W$, and the property $(\star)$ means $x$ can be {\em attained} by a {\em relatively compact} subset of $W$. This remark and the definition of sequential spaces motivate us to introduce the following three new classes of compact-type topological spaces.

%One of the basic theorems in Analysis is the Ascoli theorem which states that if $X$ is a $k$-space, then every compact subset of $\CC(X)$ is evenly continuous, see Theorem 3.4.20 in \cite{Eng}. In \cite{Noble}, Noble proved that every $k_\IR$-space satisfies the conclusion of the Ascoli theorem. So it is natural to consider the class of Tychonoff spaces which satisfy the conclusion of Ascoli's theorem. Following \cite{BG}, a  Tychonoff space $X$ is called an {\em Ascoli space} if every compact subset $\KK$ of $\CC(X)$  is evenly continuous, that is the map $X\times\KK \ni(x,f)\mapsto f(x)$ is continuous. In other words, $X$ is Ascoli if and only if the compact-open topology of $\CC(X)$ is Ascoli in the sense of \cite[p.45]{mcoy}. One can easily show that $X$ is Ascoli if and only if every compact subset of $\CC(X)$ is equicontinuous. For numerous results concerning the Ascoli property, see \cite{Gabr-LCS-Ascoli,Gabr-B1,Gab-LF,Gabr:weak-bar-L(X),Gabr-seq-Ascoli,GGKZ,GGKZ-2}. Let us mention the following result in which the necessity is proved in \cite{GGKZ} and the sufficiency in \cite{Gabr-B1} (for  more general assertions, see Section 5 of \cite{Gabr-seq-Ascoli}).

\begin{definition} \label{def:oca} {\em
A space $X$ is called
\begin{enumerate}
\item[$\bullet$] {\em open-compact attainable}  if for every open subset $W$ of $X$ and every $z\in \overline{W}$ there exists a compact subset $K$ of $X$ such that $z\in K\cap \overline{K\cap W}$;
\item[$\bullet$] {\em weakly open-compact attainable}  if for every non-closed open subset $W$ of $X$, there are a point $z\in \overline{W}\SM W$ and a compact subset $K$ of $X$ such that $z\in K\cap \overline{K\cap W}$;
\item[$\bullet$] {\em $\kappa$-sequential}  if for every non-closed open subset $W$ of $X$, there are a point $z\in \overline{W}\SM W$ and a sequence $\{x_n\}_{n\in\w}\subseteq W$ converging to $z$.
\end{enumerate}  }
\end{definition}

The study of spaces introduced in Definition \ref{def:oca}  and their relationships with the properties defined in Definition \ref{def:compact-type} is the main purpose of the article. The general relationships between all aforementioned notions are described in the following diagram (where FU=Fr\'{e}chet--Urysohn)

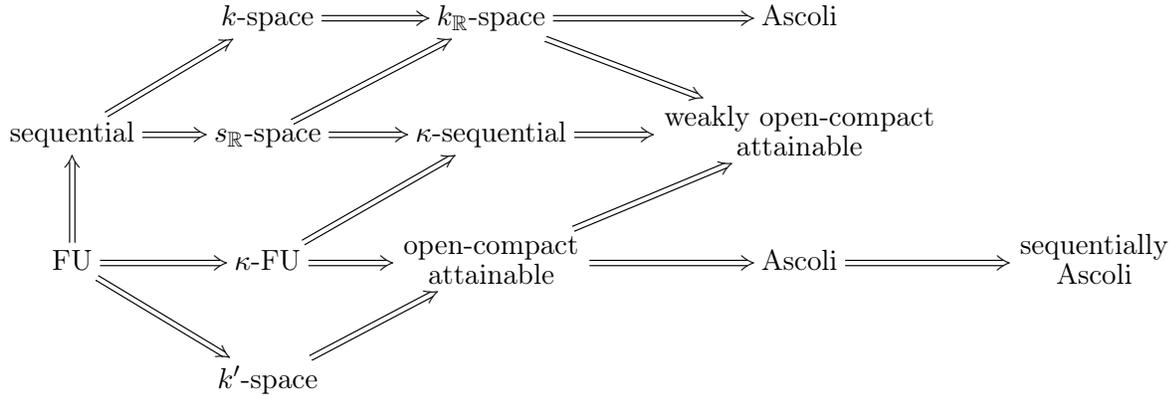
\[
\xymatrix{
& \mbox{$k$-space} \ar@{=>}[r] & \mbox{$k_\IR$-space}  \ar@{=>}[rd]  \ar@{=>}[r] & \mbox{Ascoli} & \\
\mbox{sequential} \ar@{=>}[r]\ar@{=>}[ru] & \mbox{$s_\IR$-space} \ar@{=>}[ur] \ar@{=>}[r] &  \mbox{$\kappa$-sequential} \ar@{=>}[r] & {\mbox{weakly open-compact}\atop\mbox{attainable}}&\\
\mbox{FU}\ar@{=>}[r] \ar@{=>}[rd]\ar@{=>}[u] & \mbox{$\kappa$-FU}\ar@{=>}[r] \ar@{=>}[ru] & {\mbox{open-compact}\atop\mbox{attainable}}\ar@{=>}[r]\ar@{=>}[ru] & \mbox{Ascoli} \ar@{=>}[r] & {\mbox{sequentially}\atop\mbox{Ascoli}}\\
& \mbox{$k'$-space} \ar@{=>}[ru] & &&
}
\]

Now we describe the content of the article. First of all we show in Section \ref{sec:RE} that all notions introduced in Definition \ref{def:oca} are indeed new and none of the implications in the diagram is reversible. It turns out that each pass-connected space (in particular, every topological vector space) is $\kappa$-sequential, see Proposition \ref{p:pass-connected-k-sequential}. However, each strongly zero-dimensional space $X$ which is weakly open-compact attainable must be a sequentially Ascoli space, see Proposition \ref{p:zpwoca-dim-Ascoli}.
This result and Theorem \ref{t:Cp-Ascoli} show that there are locally convex spaces which are $\kappa$-sequential but not sequentially Ascoli. On the other hand, by Proposition \ref{p:seq-Ascoli-non-woca}, each $P$-space is sequentially Ascoli but not weakly open-compact attainable.
Moreover, in Example \ref{exa:non-woca-Ascoli} we construct an Ascoli pseudocompact space which is not weakly open-compact attainable, and in Example \ref{exa:woca-non-Ascoli} we show that there are $\kappa$-sequential pseudocompact spaces which are not (sequentially) Ascoli. In Proposition \ref{p:wOCA-k-space} we show that each $k_\IR$-space $X$ is weakly open-compact attainable, and every $s_\IR$-space $X$ is $\kappa$-sequential. Another distinguish example is given in Example \ref{exa:S2-lambda}, which generalizes the classical Arens space $S_2$.

In  Section \ref{sec:permanent} we study permanent properties of the introduced notions. In Example \ref{exa:square-woca} we show that there is a Fr\'{e}chet--Urysohn space whose square is not weakly open-compact attainable, and hence none of the notions from Definition \ref{def:oca} is preserved under taking finite products. The properties of being (weakly) open-compact attainable and $\kappa$-sequential are neither closed hereditary nor preserved under taking superspaces. In Theorem \ref{t:woca-hereditary} we prove that a hereditary weakly open-compact attainable space $X$ is Fr\'{e}chet--Urysohn.

Recall that a map $f:X\to Y$ between spaces $X$ and $Y$ is called {\em pseudo open} (resp., {\em $\kappa$-pseudo open}) if whenever $f^{-1}(y)\subseteq U$ where $y\in Y$ and $U$ is open in $X$, then $y\in \Int\big(f(U)\big)$ (resp., $y\in \Int\big(\overline{f(U)}\big)$), where $\Int(A)$ denotes the interior of a set $A$. The meaning of these classes of maps is defined by the following characterizations of Fr\'{e}chet--Urysohn spaces and $\kappa$-Fr\'{e}chet--Urysohn spaces obtained by Arhangel'skii \cite{Arhan63} and Liu and Ludwig \cite{LiL}, respectively.
\begin{theorem} \label{t:FU-kFU}
Let $X$ be a space. Then:
\begin{enumerate}
\item[{\rm(i)}] $X$ is Fr\'{e}chet--Urysohn if, and only if, $X$ is a pseudo open image of a metric space.
\item[{\rm(ii)}] $X$ is  $\kappa$-Fr\'{e}chet--Urysohn if, and only if, $X$ is a $\kappa$-pseudo open image of a metric space.
\end{enumerate}
\end{theorem}

%In Definition \ref{def:func-woca} we introduce a new class of continuous maps, namely, the class of weakly $\kappa$-pseudo open maps between two topological spaces. In Lemmas \ref{l:charac-dagger-map} and \ref{l:charac-dagger-map-weak} we characterize $\kappa$-pseudo open maps  and weakly $\kappa$-pseudo open maps, using of which we show in Proposition \ref{p:composition-dagger} that the composition of (weakly) $\kappa$-pseudo open maps is (weakly) $\kappa$-pseudo open. In Proposition \ref{p:quot-oca} it is shown that open-compact attainable spaces and $\kappa$-Fr\'{e}chet--Urysohn spaces are preserved under taking  $\kappa$-pseudo open map images, and in Proposition \ref{p:quotient-woca} we show that the class of weakly open-compact attainable spaces and the class of $\kappa$-sequential spaces are stable under taking  weakly  $\kappa$-pseudo open map images.

In Section \ref{sec:maps}
we introduce a new class of continuous maps, namely, the class of weakly $\kappa$-pseudo open maps between two topological spaces (see Definition \ref{def:func-woca}). In Lemmas \ref{l:charac-dagger-map} and \ref{l:charac-dagger-map-weak} we characterize $\kappa$-pseudo open maps  and weakly $\kappa$-pseudo open maps, using of which we show in Proposition \ref{p:composition-dagger} that the composition of (weakly) $\kappa$-pseudo open maps is (weakly) $\kappa$-pseudo open. In Proposition \ref{p:quot-oca} it is shown that open-compact attainable spaces and $\kappa$-Fr\'{e}chet--Urysohn spaces are preserved under taking  $\kappa$-pseudo open map images, and in Proposition \ref{p:quotient-woca} we show that the class of weakly open-compact attainable spaces and the class of $\kappa$-sequential spaces are stable under taking  weakly  $\kappa$-pseudo open map images.
\smallskip

In Section \ref{sec:character-oca}, using \dag-maps and weakly \dag-maps defined in Section \ref{sec:permanent}, we characterize  open-compact attainable spaces, weakly open-compact attainable spaces, $\kappa$-sequential spaces and give several new characterizations of $\kappa$-Fr\'{e}chet--Urysohn spaces complementing and strengthening (ii) of Theorem \ref{t:FU-kFU}.
\smallskip

In the last Section \ref{sec:LCG-LCA-Bohr} we show in Theorem \ref{t:feathered-group-kFU} that any feathered topological group is  $\kappa$-Fr\'{e}chet--Urysohn. Hence any locally compact group is $\kappa$-Fr\'{e}chet--Urysohn, see Corollary \ref{c:lcg-kFU}. In Theorems \ref{t:LCA-Bohr-kFU} and \ref{t:LCA-Bohr-woca} we characterize locally compact abelian groups $G$ such that $G$ endowed with the Bohr topology has one of the properties from the diagram.

%-compact attainable spaces are exactly those spaces which satisfy the condition $(\star)$ from Theorem 2.5 of  \cite{Gabr-B1}, in which it is proved that any open-compact attainable space is an Ascoli space.  On the other hand, it is clear that each $\kappa$-Fr\'{e}chet--Urysohn space is open-compact attainable. Weakly open-compact attainable spaces and $\kappa$-sequential space are new. It is clear that every open-compact attainable space is weakly open-compact attainable. But the converse is not true in general. Indeed, since any $k$-space is weakly open-compact attainable (Proposition  \ref{p:wOCA-k-space}),  the Arens space and the space $\varphi$  are weakly open-compact attainable but not  open-compact attainable. It is clear that any compact space is open-compact attainable. In particular, the compact space $\beta\NN$ is open-compact attainable but not $\kappa$-Fr\'{e}chet--Urysohn since $\beta \NN$ does not contain non-trivial convergent sequences. On the other hand, below we show that the Arens space $S_\infty$ being a sequential space is not open-compact attainable.

%%%%%%%%%%%%%%%%%%%%%%%%%%%%%
%%%%%%%%%%%%%%%%%%%%%%%%%%%%%
%%%%%%%%%%%%%%%%%%%%%%%%%%%%%
%%%%%%%%%%%%%%%%%%%%%%%%%%%%%
%%%%%%%%%%%%%%%%%%%%%%%%%%%%%

\section{General relationships and examples} \label{sec:RE}

%%%%%%%%%%%%%%%%%%%%%%%%%%%%%
%%%%%%%%%%%%%%%%%%%%%%%%%%%%%
%%%%%%%%%%%%%%%%%%%%%%%%%%%%%
%%%%%%%%%%%%%%%%%%%%%%%%%%%%%
%%%%%%%%%%%%%%%%%%%%%%%%%%%%%

In this section we establish general relationships between the new notions introduced  in Definition \ref{def:oca} and other notions from the diagram in the introduction by proving some results and giving (counter)examples.

We start with the next proposition which allows us to construct in Example \ref{exa:OCA-varphi} a sequential space which is not open-compact attainable.  %and \ref{exa:OCA-Arens} sequential spaces which are not open-compact attainable.
\begin{proposition} \label{p:OCA-kappaFU}
Let $X$ be a space such that each compact subset of $X$ is Fr\'{e}chet--Urysohn (for example, $X$ is angelic). Then:
\begin{enumerate}
\item[{\rm(i)}] $X$ is open-compact attainable if, and only if, it is $\kappa$-Fr\'{e}chet--Urysohn;
\item[{\rm(ii)}] $X$ is weakly open-compact attainable if, and only if, it is $\kappa$-sequential.
\end{enumerate}
\end{proposition}

\begin{proof}
(i) Assume that $X$ is open-compact attainable. To show that $X$ is  $\kappa$-Fr\'{e}chet--Urysohn, let $U$ be an open subset of $X$ and $x\in \overline{U}$. If $x\in U$, then the constant sequence $\{x\}\subseteq U$ converges to $x$. Assume that $x\not\in U$. Since $X$ is open-compact attainable, there exists $L\subseteq U$ such that $\overline{L}$ is compact and $x\in \overline{L}$. By assumption $\overline{L}$ is Fr\'{e}chet--Urysohn, and hence there is a sequence $\{x_n\}_{n\in\w}$ in $L$ which converges to $x$. As $L\subseteq U$, we also have $\{x_n\}_{n\in\w}\subseteq U$. Thus $X$ is a $\kappa$-Fr\'{e}chet--Urysohn space.
\smallskip

(ii) Assume that $X$ is weakly open-compact attainable.  Let $W$ be a non-closed open subset of $X$. Then there are a point $z\in \overline{W}\SM W$ and a compact subset $K$ of $X$ such that $z\in K\cap \overline{K\cap W}$. Since $K$ is Fr\'{e}chet--Urysohn and $z\in\overline{K\cap W}$, there is a sequence  $\{x_n\}_{n\in\w}\subseteq K\cap W\subseteq W$ converging to $z$. Thus $X$ is  $\kappa$-sequential.\qed
\end{proof}

Recall that the space $\varphi$ is the direct locally convex sum $\bigoplus_{n\in\w} E_n$ with $E_n=\mathbf{F}$ for all $n\in\w$. It is well known that $\varphi$ is a sequential non-Fr\'{e}chet--Urysohn space, see Example 1 of \cite{nyikos}. Below we strengthen the negative part of this assertion.

\begin{example}  \label{exa:OCA-varphi}
$\varphi$ is a sequential non-open-compact attainable space.
\end{example}

\begin{proof}
It is well-known that any bounded subspace of $\varphi$ is finite-dimensional and hence metrizable. Since, by Proposition 2.9 of \cite{Gabr-B1}, $\varphi$ is not $\kappa$-Fr\'{e}chet--Urysohn, Proposition \ref{p:OCA-kappaFU} implies that $\varphi$ is also not open-compact attainable.\qed
\end{proof}

The Arens space $S_2$ is one of the most important and famous sequential spaces which is not Fr\'{e}chet--Urysohn, see Example 1.6.19 of \cite{Eng}. Below we naturally generalize  $S_2$.

Let $\lambda$ be an infinite ordinal with its standard ordered topology. We denote by $S_2(\lambda)$ the union
\[
S_2(\lambda):=  (\lambda+1) \cup (\NN\times\lambda)
\]
The set $S_2(\lambda)$ is topologized as follows. Each point of $\NN\times\lambda$ is isolated; a neighborhood base at each $\alpha\leq\lambda$ is the family
$
(\gamma,\alpha]\cup \bigcup_{\gamma<i\leq \alpha} V_{n_i,i}
$
for some $\gamma\in \alpha$ and $n_i\in\NN$, where
\[
V_{n_i,i}:=\big\{(k,i)\in\NN\times\{i\}: k\geq n_i\big\} \; \mbox{ if $i<\lambda$, and }\; V_{n_i,i}:=\{\lambda\} \mbox{ if $i=\lambda$}.
\]
It is easy to see that $S_2(\lambda)$ is a Tychonoff space. By definition $S_2(\w)=S_2$.

Recall that $V(\lambda)$ denotes the Fr\'{e}chet--Urysohn fan of cardinality $\lambda$, that is,
\[
V(\lambda):=\{x_\infty\}\cup \{x_{n,\alpha}: n\in\w, \alpha\in\lambda\},
\]
where the point $x_\infty$ is the unique non-isolated point in $V(\lambda)$, and a subset $U\subseteq V(\lambda)$, $x_\infty\in U$, is a neighborhood of $x_\infty$ if and only if $|\{v\in\w: x_{n,\alpha}\notin U\}|<\w$ for each $\alpha<\lambda$. Recall also that $x_{n,\alpha}\to x_\infty$ for every $\alpha\in\lambda$, and for every compact subset $K$ of $V(\lambda)$, there is a finite subset $F$ of $\lambda$ such that
\[
K\subseteq \{x_\infty\}\cup \{x_{n,\alpha}: n\in\w, \alpha\in F\}.
\]

By Example 2.4 in \cite{LiL}, $S_2$ is not $\kappa$-Fr\'{e}chet--Urysohn. In (iv) of the next example we strengthen this result.

\begin{example} \label{exa:S2-lambda}
Let $\lambda$ be an infinite ordinal.
\begin{enumerate}
\item[{\rm(i)}] For every compact subset $K$ of  $S_2(\lambda)$ there is a finite subset $F$ of $\lambda$ such that
\[
K\subseteq \big(\NN \times F\big) \cup (\lambda+1).
\]
\item[{\rm(ii)}] The map $f:S_2(\lambda)\to V(\lambda)$ defined by
\[
f(n,i)=(n,i) \mbox{ if $(n,i)\in \NN\times\lambda$, and } f(i)=i \mbox{ if $i\in \lambda+1$},
\]
is perfect.
\item[{\rm(iii)}] If $\mu\in\lambda$, then $S_2(\mu)$ is a closed subspace of $S_2(\lambda)$.
\item[{\rm(iv)}] $S_2(\lambda)$  is not an open-compact attainable space.
\item[{\rm(v)}] $S_2(\lambda)$ is a $k$-space and hence an Ascoli weakly open-compact attainable space.
\item[{\rm(vi)}] If $\lambda$ is countable, then $S_2(\lambda)^n$  is a sequential  space for every $n\in\NN$.
\item[{\rm(vii)}] If $\lambda\geq\w_1$, then $S_2(\lambda)$  is not a $\kappa$-sequential space.
\end{enumerate}
\end{example}

\begin{proof}
(i) Let $I:=\pr_{\lambda} \big(K\SM (\lambda+1)\big)$ be the projection of $K\SM (\lambda+1)$ onto  $\lambda$. To prove (i) we show that $I$ is finite. Suppose for a contradiction that $I$ is infinite. Take a sequence $\alpha_0<\alpha_1<\cdots$ in $I$. For every $k\in\w$, choose a point $(m_k,\alpha_k)\in K$. Then the sequence $S=\{(m_k,\alpha_k)\}_{k\in\w}$ is discrete in the compact set $K$. Therefore to get a contradiction it suffices to show that $S$ is closed.

Let $x\in S_2(\lambda)\SM S$. Assume that $x\leq\lambda$. If $x\not\in \{\alpha_k\}_{k\in\w}$, then the neighborhood $\{x\}\cup\{n\in\w: n\in\w\}$ of $x$ does not intersect $S$. Define a standard neighborhood $U$ of $x$ by
\[
U=(0,x]\cup \bigcup_{0<\alpha\leq x} \big\{(i,\alpha)\in\NN\times\{\alpha\}: i\geq n_\alpha\big\}\;\; \mbox{ if $x\in\lambda$},
\]
or
\[
U=(0,x]\cup \bigcup_{0<\alpha< x} \big\{(i,\alpha)\in\NN\times\{\alpha\}: i\geq n_\alpha\big\}\;\; \mbox{ if $x=\lambda$},
\]
where $n_{\alpha_k}>m_k$ and $n_\alpha=1$ if $\alpha\not\in \{\alpha_k\}_{k\in\w}$. Then $U\cap S=\emptyset$, and hence $x\not\in \overline{S}$.
If $x\in (\NN\times \lambda)\SM S$, then $x$ is isolated, and hence $x\not\in \overline{S}$. These cases show that $S$ is closed.
\smallskip

(ii) Recall that $V(\lambda)$ is a Fr\'{e}chet--Urysohn space and hence a $k$-space. Therefore, by Theorem 3.7.17 of \cite{Eng}, the map $f:S_2(\lambda)\to V(\lambda)$ is perfect if (and only if) for every compact subset $K$ of  $V(\lambda)$, the restriction of $f$ onto $f^{-1}(K)$ is perfect.

Recall also that $K$ sits in $\{x_\infty\} \cup \bigcup_{i\in F} S_i$ for some finite $F\subseteq \lambda$. Therefore,   $f^{-1}(K)$ is a closed subset of $C:=(\lambda+1) \cup \bigcup_{i\in F} \{ (n,i): n\in\NN\}$. By (i), $C$ is a compact subset of $S_2(\lambda)$. It follows that $f$ is perfect.
\smallskip

(iii) holds true by definition.
\smallskip

(iv) To show that  $S_2(\lambda)$  is not open-compact attainable, we consider the open subset  $U=\NN\times\NN$ of $S_2(\lambda)$  and observe that $\w\in \overline{U}\SM U$. Let now $K$ be a compact subset of $S_2(\lambda)$. Then, by (i),  $U\cap K$ sits in $\NN\times \{0,\dots,m\}$ for some $m\in\w$. It is clear that $\w\not\in \overline{U\cap K}$. Thus $S_2(\lambda)$ is not an open-compact attainable space.
\smallskip

%To show that $S_2(\lambda)$ is not Fr\'{e}chet--Urysohn, by (ii), it suffices to check that $S_2=S_2(\w)$ is not Fr\'{e}chet--Urysohn. This immediately follows from the fact that the subset $\{(n,m): n,m\in\NN\}$ of $S_2$ has $\w$ as a cluster point but it has no a convergent sequence converging to $\w$.

(v) Since $V(\lambda)$ is Fr\'{e}chet--Urysohn and hence a $k$-space and since $f$ is perfect by (ii), Theorem 3.7.25 of \cite{Eng} implies that  $S_2(\lambda)$ is  a $k$-space.

(vi) Assume that $\lambda$ is countable. By Theorem 3.7.7 of \cite{Eng}, $f^n$ is also a perfect map. Since $\lambda$ is countable, $V(\lambda)$ is a $k_\w$-space. Therefore, by (4) of \cite{Franklin-Smith}, it follows that  $V(\lambda)^n$ is a $k_\w$-space for every $n\in\NN$. These two facts and Theorem 3.7.25 of \cite{Eng} imply that also $S_2(\lambda)^n$ is a $k$-space. As $S_2(\lambda)^n$ is countable, each of its compact subset is metrizable. Thus $S_2(\lambda)^n$ is a sequential space.
\smallskip

(vii) Let $W=\big(\NN\times[0,\w_1)\big)\cup [0,\w_1)$. Then $W$ is open in $S_2(\lambda)$ and
\[
\overline{W}\subseteq W\cup\big(\NN\times\{\w_1\}\big) \mbox{ if } \lambda>\w_1, \mbox{ and } \overline{W}=W\cup\{\w_1\} \mbox{ if } \lambda=\w_1.
\]
Since the cofinality of $\w_1$ is $\w_1$, it is clear that $W$ has no a sequence which converges to a point of $\overline{W}\SM W$. Thus $S_2(\lambda)$  is not $\kappa$-sequential.\qed
\end{proof}
In particular, by (iv) and (vi) of Example \ref{exa:S2-lambda}, the Arens space $S_2$ is a sequential non-open-compact attainable space.

The next assertion is evidently follows from the definition of weakly open-compact attainable spaces.
\begin{proposition} \label{p:non-woca}
If a space $X$ contains an open non-closed subset $U$ whose intersection with any compact subset of $X$ is closed is not weakly open-compact attainable.
In particular, if $X$ is a non-discrete space whose compact subsets are finite, then $X$  is not weakly open-compact attainable.
\end{proposition}

%The Arens space $S_2$ is one of the most important and famous sequential spaces which is not Fr\'{e}chet--Urysohn, see Example 1.6.19 of \cite{Eng}. Let us recall the definition of $S_2$. The space $S_2$ is the union $S_2=\{\infty\}\cup \NN\cup (\NN\times\NN)$ in which all points of $\NN\times\NN$ are isolated. A neighborhood base at each $n\in\NN$ is the family $U_{n,m}:=\{n\}\cup\{(n,j)\in\NN\times\NN: j\geq m\}$ for $m\in\NN$, and a neighborhood base at $\infty$ is the family
%$
%\{\infty\}\cup \bigcup_{n\geq a} U_{n,m_n}
%$
%for some $a, m_n\in\NN$. It is known (see Example 2.4 in \cite{LiL}) that the Arens space is not $\kappa$-Fr\'{e}chet--Urysohn. Below we strengthen this result.
%\begin{example}  \label{exa:OCA-Arens}
%The Arens space $S_2$ is a sequential non-open-compact attainable space.
%\end{example}

%\begin{proof}
%By the definition of $S_2$, the subset $U=\NN\times\NN$ is open. It is easy to see (see the proof of  Example 1.6.19 of \cite{Eng}) that $U$ does not contain infinite compact sets and $\infty\in \overline{U}$. By Proposition \ref{p:non-woca}, $S_2$ is not open-compact attainable.\qed
%\end{proof}

\begin{remark} {\em
(i) Let $Y=\NN\cup\{x_\ast\}$ where $x_\ast\in \beta\NN\SM \NN$, and let $X=\IR^\mathfrak{c}$. We consider $Y$ as a closed subset of $X$. Then $X$ is a $\kappa$-Fr\'{e}chet--Urysohn space (as a product of $\IR$, see \cite{LiL}). Since $Y$ has no infinite compact sets and is closed in $X$, it follows that $X$ is not a $k'$-space.

(ii) Clearly, the space $\beta\NN$ is a $k'$-space. But since $\beta\NN$ has no non-trivial convergent sequences it is not $\kappa$-Fr\'{e}chet--Urysohn.}
\end{remark}

%\begin{proposition} \label{p:lc-oca}
%Each locally compact space $X$ is open-compact  attainable.
%\end{proposition}

%\begin{proof}
%Let $W$ be a non-closed open subset of $X$, and let $x\in \overline{W}\SM W$. Take a compact neighborhood $K$ of $x$. Then $x\in K\cap \overline{K\cap W}$. Thus $X$ is open-compact  attainable.\qed
%\end{proof}

Now we consider relationships between $k_\IR$-spaces and $s_\IR$-spaces and weakly open-compact attainable spaces and  $\kappa$-sequential spaces, respectively.

\begin{proposition} \label{p:wOCA-k-space}
\begin{enumerate}
\item[{\rm(i)}] Each $k_\IR$-space $X$ is weakly open-compact attainable.
%\item[{\rm(ii)}] Each sequential space $X$ is weakly $\kappa$-Fr\'{e}chet--Urysohn.
\item[{\rm(ii)}] Each $s_\IR$-space $X$ is $\kappa$-sequential.
\item[{\rm(iii)}] Each $k'$-space is is open-compact  attainable. In particular, every locally compact space $X$ is open-compact attainable. %Each locally compact space $X$ is open-compact  attainable.
\item[{\rm(iv)}] Each $\kappa$-Fr\'{e}chet--Urysohn space is open-compact attainable.
\item[{\rm(v)}] Each open-compact attainable space $X$ is an Ascoli space.
\end{enumerate}
\end{proposition}

\begin{proof}
(i) Let $W$ be a non-closed open subset of $X$, and suppose for a contradiction that for every $z\in \overline{W}\SM W$ and each compact subset $K$ of $X$, we have $z\not\in K\cap \overline{K\cap W}$. Then $K\cap W$ is closed (otherwise, any point $z\in \overline{K\cap W}\SM (K\cap W)$ belongs to $K\cap \overline{W}$, and hence $z\in \overline{W}\SM W$ and $z\in K\cap \overline{K\cap W}$). Therefore, for each compact subset $K$ of $X$, $K\SM W$ and $K\cap W$ are disjoint compact sets. Then the characteristic function $\mathbf{1}_W$ is $k$-continuous and discontinuous. Therefore, $X$ is not a $k_\IR$-space, a contradiction.

%Let $W$ be a non-closed open subset of $X$. Since $X$ is a $k$-space, there is a compact subset $K'$ of $X$ such that $W\cap K'$ is not closed in $K'$. Take an arbitrary point $z\in \overline{W\cap K'}\SM (W\cap K')$ and set $K:=\overline{W\cap K'}$. Then $K$ is a compact subset of $X$ such that $z\in K\cap \overline{K\cap W}$.

%(ii) Let $W$ be a non-closed open subset of $X$. Since $X$ is a sequential space, there is a sequence $\{x_n\}_{n\in\w}$ in $W$ which converges to some point $x\in \overline{W}\SM W$. Thus $X$ is a weakly $\kappa$-Fr\'{e}chet--Urysohn space.

(ii) Suppose for a contradiction that $X$ is not a $\kappa$-sequential space. Then there is an open non-closed subset $W$ of $X$ such that there is no sequence in $W$ converging to a point of $\overline{W}\SM W$. Define a function $f:X\to \IR$ by
\[
f(x)=1 \;\mbox{ if } x\in W, \; \mbox{ and } \; f(x)=0 \;\mbox{ if } x\in X\SM W.
\]
By the assumption on $W$, we obtain that $f$  is sequentially continuous. Since $W$ is not closed, $f$ is not continuous and hence $X$ is not an  $s_\IR$-space. This contradiction finishes the proof.

(iii) and (iv) are  obvious. %Let $W$ be a non-closed open subset of $X$, and let $x\in \overline{W}\SM W$. Take a compact neighborhood $K$ of $x$. Then $x\in K\cap \overline{K\cap W}$. Thus $X$ is open-compact  attainable.

(v) is Theorem 2.5 of \cite{Gabr-B1}.
\qed
\end{proof}

%Recall that a space $X$ is of {\em pointwise countable type} (or simply of {\em point-countable type}) provided that every  point of the space is contained in a compact set $K\subseteq X$ of countable character. Any \v Chech complete space, first countable space or locally compact space is of point-countable type, see \cite[3.1.E(b), 3.3.I]{Eng}.

%\begin{problem}
%Characterize \v{C}ech-complete spaces which are  open-compact  attainable.
%\end{problem}
%}

Let $X=\prod_{\alpha\in \AAA}X_\alpha$ be a product of spaces. For $x=(x_\alpha)_\alpha, y=(y_\alpha)_\alpha\in X$, let
\[
\begin{aligned}
\delta(x,y) &= \{ \alpha\in \AAA : x_\alpha\neq y_\alpha\},\\
\Sigma(X,x) &= \{ y\in X : |\delta(x,y)|\leq\w\}.
\end{aligned}
\]
The set $\Sigma(X,x)$ is called the {\em $\Sigma$-product in $X$ around the point $x$}.

The next example also shows that there are Ascoli pseudocompact spaces which are not weakly open-compact attainable.

\begin{example} \label{exa:non-woca-Ascoli}
Let $X=S\cup\{\mathbf{1}\}$, where $S$ is the $\Sigma$-product in $[0,1]^{\w_2}$ around the zero function $\mathbf{0}$ and $\mathbf{1}$ is the unit function on $\w_2$. Then:
\begin{enumerate}
\item[{\rm(i)}] $S$ is a dense Fr\'{e}chet--Urysohn, sequentially compact and $\w$-bounded subspace of $X$;
\item[{\rm(ii)}]  $X$ is an Ascoli pseudocompact space which is not weakly open-compact attainable.
\end{enumerate}
\end{example}

\begin{proof}
(i) is well-known, see Exercise 3.10.D of \cite{Eng}.

(ii) Since $S$ is a dense $\w$-bounded subset of $X$, the space $X$ is Ascoli by Proposition 2.2 of \cite{Gabr-pseudo-Ascoli}. On the other hand, $S$ is open and $\mathbf{1}\in \overline{S}$. So, to prove that $X$ is not weakly open-compact attainable it suffices to show that there is no compact subset $K$ of $X$ such that $\mathbf{1}\in K\cap \overline{K\cap S}$. To this end,  we show that $F:=K\cap S$ is closed hence compact in $X$ (see Proposition \ref{p:non-woca}). %$\mathbf{1}\not\in \overline{K\cap S}$.

Suppose for a contradiction that $F$ is not closed. Then $\mathbf{1}\in \overline{F}$.  For every $\alpha<\w_1 <\w_2$, we consider  the natural projection $\pi_\alpha:[0,1]^{\w_2} \to [0,1]^\alpha$. Since $S$ is sequentially compact and $K$ is compact, it follows that $F$ is sequentially compact. As $\alpha$ is countable, we obtain that $\pi_\alpha(F)$ is a sequentially compact subset of the metrizable compact space $[0,1]^\alpha$. Therefore $\pi_\alpha(F)$ is compact, and hence $\pi_\alpha(\mathbf{1})\in \pi_\alpha(F)=\pi_\alpha(\overline{F})$. Choose $q_\alpha\in F$ such that $\pi_\alpha(\mathbf{1})=\pi_\alpha(q_\alpha)$.

Consider the set $Q:=\{q_\alpha: \alpha<\w_1\}\subseteq F$. As $Q\subseteq K$, the set $Q$ has a cluster point $y=(y_i)\in K$. Taking into account that $\pi_\alpha(\mathbf{1})=\pi_\alpha(q_\alpha)$ for every $\alpha<\w_1$, we obtain that $\pi_{\w_1}(\mathbf{1})=\pi_{\w_1}(y)$. Therefore $y$ has at least $\w_1$ non-zero coordinates and hence $y=\mathbf{1}$. On the other hand, the support $A$ of $Q$ has cardinality $\w_1$. As $\w_1 <\w_2$, we obtain that $y_i=0$ for every $i\in\w_2\SM A$. Therefore $y\not=\mathbf{1}$, a contradiction.\qed
\end{proof}

Recall that a space $X$ is called a {\em $P$-space} if  the intersection of any countably many open sets is open.

\begin{proposition} \label{p:seq-Ascoli-non-woca}
Each non-discrete $P$-space $X$ is sequentially Ascoli but not weakly open-compact attainable.
\end{proposition}

\begin{proof}
The space $X$ is sequentially Ascoli by Proposition 2.9 of  \cite{Gabr:weak-bar-L(X)}. Since $P$-spaces has no infinite compact subsets (see Lemma 2.8 of \cite{Gabr:weak-bar-L(X)}) and $X$ is not discrete, Proposition \ref{p:non-woca} implies that $X$ is not weakly open-compact attainable.\qed
\end{proof}

The next proposition shows that $\kappa$-sequentiality and weak open-compact attainability are not of interest in the class of pass-connected spaces.

\begin{proposition}\label{p:pass-connected-k-sequential}
Each path-connected space $X$ is $\kappa$-sequential.
\end{proposition}
\begin{proof}
Let $U\subseteq X$ be an open subset such that $\overline{U}\setminus U\neq\emptyset$.
Let $x_0\in U$ and $x_1\in X\setminus U$. Take a continuous path  $f: [0,1]\to X$ such that $f(0)=x_0$ and $f(1)=x_1$. Then $V=f^{-1}(U)$ is an open subset of $[0,1]$, $0\in V$ and $1\notin V$. Therefore there exists a sequence $(t_n)_n\subseteq V$ which converges to some point $t\in [0,1]\setminus V$. Let $y_n=f(t_n)$ and $y=f(t)$. Then the sequence $(y_n)_n\subseteq U$ converges to the point $y\in X\setminus U$. Thus $X$ is a $\kappa$-sequential space.\qed
\end{proof}

\begin{corollary} \label{c:lcs-k-sequential}
Each topological vector space is $\kappa$-sequential.
\end{corollary}

%Corollary \ref{c:lcs-k-sequential} and Theorem \ref{t:Cp-Ascoli} imply the next  corollary.
%\begin{corollary} \label{c:Cp-k-sequential-non-Ascoli}
%If a space $X$ does not have the property $(\kappa)$, then $C_p(X)$ is a $\kappa$-sequential space which is not sequentially Ascoli.
%\end{corollary}

Theorem \ref{t:Cp-Ascoli} and (iv) and (v) of Proposition \ref{p:wOCA-k-space} immediately imply
\begin{corollary} \label{c:Cp-oca}
$C_p(X)$ is an open-compact attainable space if, and only if, $C_p(X)$ is $\kappa$-Fr\'{e}chet--Urysohn.
\end{corollary}

Denote by $B_w$ the closed unit ball of a Banach space $E$ endowed with the weak topology.
\begin{corollary} \label{c:unit-ball-oca}
Let $B$ be the closed unit ball of a Banach space $E$. Then $B_w$ is $\kappa$-sequential and the following assertions are equivalent:
\begin{enumerate}
\item[{\rm(i)}] $B_w$ is open-compact attainable;
\item[{\rm(ii)}] $B_w$ is  Fr\'{e}chet--Urysohn;
\item[{\rm(iii)}] $E$ has no an isomorphic copy of $\ell_1$.
\end{enumerate}
\end{corollary}

\begin{proof}
Since $B_w$ is pass-connected, it is $\kappa$-sequential by Proposition \ref{p:pass-connected-k-sequential}.

Theorem 1.9 of \cite{GKP} states that the condition ``$B_w$ is an Ascoli space'' is equivalent to the conditions (ii) and (iii). Now the equivalences (i)$\LRa$(ii)$\LRa$(iii) immediately follow from the fact that the open-compact attainability lies between Fr\'{e}chet--Urysohness and the property of being an Ascoli space.\qed
%Recall that each Fr\'{e}chet--Urysohn space is open-compact attainable, and any  open-compact attainable  space is an Ascoli space. Now the equivalences (i)$\LRa$(ii)$\LRa$(iii) immediately follow from Theorem 1.9 of \cite{GKP}.\qed
\end{proof}

Let $X$ be a topological space. A family  $\{ A_i\}_{i\in I}$ of subsets of  $X$ is called {\em compact-finite} if for every compact $K\subseteq X$, the set $\{ i\in I: K\cap A_i\}$ is finite; and $\{ A_i\}_{i\in I}$ is called {\em locally finite} if for every point $x\in X$ there is a neighborhood $X$ of $x$ such that the set $\{ i\in I: U\cap A_i\}$ is finite. A family  $\{ B_i\}_{i\in I}$ of subsets of  $X$ is called {\em strongly} ({\em functionally}) {\em compact-finite} if for every $i\in I$, there is an (functionally) open neighborhood  $U_i$ of $B_i$ such that the family $\{ U_i\}_{i\in I}$ is compact-finite.

In Proposition \ref{p:pass-connected-k-sequential} we considered pass-connected spaces. In the opposite case when $X$ is a zero-dimensional space we have the following assertion.

\begin{proposition} \label{p:zpwoca-dim-Ascoli}
Let $X$ be a strongly zero-dimensional space. If $X$ is weakly open-compact attainable, then $X$ is a sequentially Ascoli space.
\end{proposition}

\begin{proof}
Let $\{F_n\}_{n\in\w}$ be a strongly functionally compact-finite sequence of functionally closed subsets of $X$. To show that $X$ is sequentially Ascoli, by Theorem 2.7 of \cite{Gabr-seq-Ascoli}, we have  to prove that  $\{F_n\}_{n\in\w}$  is locally finite.

Let $\{U_n\}_{n\in\w}$  be a compact-finite sequence of functionally open subsets of $X$ such that $F_n\subseteq U_n$ for every $n\in\w$.
Since $X$ is strongly zero-dimensional, for every $n\in\w$, by Lemma 6.2.2 of \cite{Eng},  there exists a clopen $W_n\subseteq X$ such that $F_n\subseteq W_n \subseteq  U_n$. Clearly,  the sequence $\{W_n\}_{n\in\w}$  is also compact-finite. Set $W:=\bigcup_{n\in\w} W_n$.

Let us show that even the sequence $\{W_n\}_{n\in\w}$  is locally finite. Suppose for a contradiction that $\{W_n\}_{n\in\w}$  has accumulation points. Let $A$ be the set of all  accumulation points of $\{W_n\}_{n\in\w}$. By definition, $A$ is also the set of accumulation points of $\{W_n\}_{n> m}$ for every $m\in\w$. Observe also that $\overline{W}\SM W\subseteq A$. Indeed, if  $x\in \overline{W}\SM W$ and the set $I_x:=\{n\in\w: \mathcal{O}\cap W_n\not=\emptyset\}$ is finite for some clopen neighborhood $\mathcal{O}$ of $x$, then $\mathcal{O}\cap \overline{\bigcup_{n\in \w\SM I_x} W_n}=\emptyset$ and hence $x\in \overline{\bigcup_{n\in I_x} W_n}=\bigcup_{n\in I_x} W_n\subseteq W$ which is impossible. On the other hand, since $\{W_n\}_{n\in\w}$  is compact-finite, for every $a\in A$, there is $m_a\in\w$ such that $a\not\in \bigcup_{n> m_a} W_n$ and $a$ is an accumulation point of $\{W_n\}_{n> m_a}$. Therefore, passing to a subsequence of the form  $\{W_n\}_{n\geq m}$ if needed, we can assume that $\overline{W}\SM W\not=\emptyset$ and $\overline{W}\SM W\subseteq A$.

Since $X$ is weakly open-compact attainable, there are a point $x\in \overline{W}\SM W$  and a compact subset $K$ of $X$ such that $x\in K\cap \overline{K\cap W}$. Set $J:=\{n\in \w: W_n\cap K\neq\emptyset\}$. If $J$ is finite, then (recall that all $W_n$ are clopen)
\[
\overline{K\cap W}=\overline{K\cap \bigcup_{n\in J} W_n}=  \bigcup_{n\in J} K\cap W_n \subseteq W
\]
and hence $x\in W$. This contradiction shows that $J$ must be infinite. But in this case we obtain that the sequence  $\{W_n\}_{n\in\w}$  is not compact-finite, a contradiction.\qed
\end{proof}

Under addition assumption that $X$ is a perfectly normal space we can reverse the conclusion in Proposition \ref{p:zpwoca-dim-Ascoli}.
Recall that a space $X$ is called {\em perfectly normal} if for every two disjoint nonempty closed sets $E$ and $F$, there is a continuous function $f:X\to[0,1]$ such that $f^{-1}(0)=E$ and $f^{-1}(1)=F$.

\begin{proposition} \label{p:zpwoca-dim-Ascoli-perfectly-normal}
Let $X$ be a strongly zero-dimensional perfectly normal space. Then $X$ is weakly open-compact attainable if, and only if, $X$ is a sequentially Ascoli space.
\end{proposition}

\begin{proof}
The necessity follows from Proposition \ref{p:zpwoca-dim-Ascoli}. To prove the sufficiency, suppose for a contradiction that $X$ is not weakly open-compact attainable. Then there is a non-closed open subset $W$ of $X$ such that for any point $z\in \overline{W}\SM W$ and each compact subset $K$ of $X$, we have $z\not\in K\cap \overline{K\cap W}$. In particular, $K\cap W$ is closed in $K$ for every compact set $K\subseteq X$, that is $W$ is $k$-closed.

Since $X$ is perfectly normal, $W$ is functionally open. As $X$ is strongly zero-dimensional, there exists a clopen partition $\{W_n\}_{n\in\w}$ of $W$.
Since $W$ is $k$-closed, the sequence $\{W_n\}_{n\in\w}$ is compact-finite. By assumption, $X$ is sequentially Ascoli and therefore, by Theorem 2.7 of \cite{Gabr-seq-Ascoli}, the sequence  $\{W_n\}_{n\in\w}$  is locally finite. Hence $W=\bigcup_{n\in\w} W_n$ is clopen that contradicts the choice of $W$.
Thus $X$ is a weakly open-compact attainable space.\qed
\end{proof}

\begin{corollary} \label{c:counable-X-woca-Seq-Ascoli}
Let $X$ be a countable space. Then:
\begin{enumerate}
\item[{\rm(i)}] $X$ is $\kappa$-sequential if, and only if, $X$ is weakly open-compact attainable if, and only if,  $X$ is a sequentially Ascoli space;
\item[{\rm(ii)}] $X$ is $\kappa$-Fr\'{e}chet--Urysohn if, and only if, it is open-compact attainable.
\end{enumerate}
\end{corollary}

\begin{proof}
Since $X$ is countable, it is strongly zero-dimensional by Corollary 6.2.8 of \cite{Eng}. Since $X$ is hereditarily Lindel\"{o}f, $X$ is perfectly normal.
The countability of $X$ implies that any compact subset is even metrizable, see Theorem 3.1.21 of \cite{Eng}. Now the corollary follows from (ii) of Proposition \ref{p:OCA-kappaFU} and Proposition \ref{p:zpwoca-dim-Ascoli-perfectly-normal}.\qed
\end{proof}

\begin{proposition} \label{p:zpwoca-Ascoli}
Let $X$ be a zero-dimensional  pseudocompact space. If $X$ is weakly open-compact attainable, then $X$ is an Ascoli space.
\end{proposition}

\begin{proof}
Let $(U_n)_{n\in\w}$ be a disjoint sequence of nonempty open sets. As $X$ is zero-dimensional, for every $n\in\w$, there is a clopen set $V_n$ such that $V_n\subseteq U_n$. Set $V:=\bigcup_{n\in\w} V_n$. Since $X$ is pseudocompact, the set $\overline{V}\setminus V$ is not empty. As all $V_n$ are clopen,
we can repeat the corresponding part of the proof of Proposition \ref{p:zpwoca-dim-Ascoli} to show that the set $\overline{V}\setminus V$ consists of accumulation points of the sequence $(V_n)_{n\in\w}$. As $X$ is weakly open-compact attainable, there exists a compact set that intersects infinitely many members of the sequence $(V_n)_{n\in\w}$ and hence of $(U_n)_{n\in\w}$. Thus, by Theorem 1.2 of \cite{Gabr-pseudo-Ascoli}, $X$ is an Ascoli space.\qed
\end{proof}

Zero-dimensionality in Propositions \ref{p:zpwoca-dim-Ascoli} and  \ref{p:zpwoca-Ascoli}  is essential even in the class of pseudocompact spaces (which are not pass-connected) as the following example shows.

%The next example shows that there are $\kappa$-sequential pseudocompact spaces which are not sequentially Ascoli.

\begin{example} \label{exa:woca-non-Ascoli}
Let $S$ and $S_1$ be the $\Sigma$-products in $[0,1]^{\w_1}$ around the zero function $\mathbf{0}$ and the unit function $\mathbf{1}$, respectively. Set
\[
X: =\big(\w \times S\big)\cup \Big( \{\w\}\times  (\{\mathbf{0}\}\cup S_1)\Big)\subseteq (\w+1)\times [0,1]^{\w_1}.
\]
Then $X$ is a $\kappa$-sequential pseudocompact space which is not sequentially Ascoli.
\end{example}

\begin{proof}
The proof is separated onto the following three claims.
\smallskip

{\em Claim 1. The space $X$ is pseudocompact.} Indeed, let us remark first that each open nonempty set $U$ contains an open set of the form $U_n:=\big(\{n\}\times [0,1]^{\w_1}\big)\cap X$ for some $n\in\w$. Now, let $(\mathcal{O}_n)_n$ be a sequence of open nonempty sets in $X$. By the remark we can assume that $\mathcal{O}_n=\{m_n\}\times V_n$, where $V_n$ is an open subset of $S$. If the sequence $(m_n)_n$ is bounded by $N$, then the sequence $(\mathcal{O}_n)_n$ lies in the countably compact space $\{0,1,...,N\}\times S$ and, therefore, it has a limit point. If $(m_n)_n$ is unbounded, then we can assume that $(m_n)_n$ is a strictly increasing sequence of integers.

We show that the sequence $(V_n)_n$ accumulates at some point $x\in [0,1]^{\w_1}\setminus S$. For every $n\in\w$, there exists a standard open set $W_n=\prod_{\alpha<\w_1} W_{n,\alpha}$ in $[0,1]^{\w_1}$  such that $W_n\cap S\subseteq V_n$, where $W_{n,\alpha}$ is an open subset of $[0,1]$ and the set $A_n=\{\alpha<\w_1: W_{n,\alpha}\neq[0,1]\}$ is finite and nonempty. Then the set $A=\bigcup_{n<\w} A_n$ is at most countable. Let $Z$ be the set of all accumulation points of the sequence $(W_n)_n$ in $[0,1]^{\w_1}$. Then $Z=F\times [0,1]^{\w_1\setminus A}$, where $F$ is some closed nonempty subset of $[0,1]^A$.
Therefore, the set $Z\setminus S$ is not empty. Let $x\in Z\setminus S$.
Then the sequence $(V_n)_n$ accumulates at $x\in [0,1]^{\w_1}\setminus S$.

Since $m_n\to\infty$ and $(V_n)_n$ accumulates at $x\in [0,1]^{\w_1}\setminus S$, the sequence $(\mathcal{O}_n)_n$ accumulates at the point $(\w,x)\in X$. Thus $X$  is pseudocompact and Claim 1 is proved.
\smallskip

{\em Claim 2. The space $X$ is $\kappa$-sequential.} Indeed, let $U$ be an open set in $X$ such that $\overline{U}\setminus U\neq\emptyset$.
Set
\[
A:=\big\{n\in\w: U\cap \big(\{n\}\times  S\big)\neq\emptyset\big\}.
\]
%Since $\bigcup_{n\in\w} \{n\}\times  S=\w\times S$ is dense in $X$,
Consider the following three possible cases.

{\em Case 1. There exists $n\in A$ such that $\big(\{n\}\times  S\big)\cap U$ is not closed in $\{n\}\times  S$.} Since $S$ is a Fr\'{e}chet--Urysohn space (see Exercise 3.10.D of \cite{Eng}), there exists a sequence $(x_n)_n\subseteq U\cap \big(\{n\}\times  S\big)$ converging to some point in $\big(\{n\}\times  S\big)\setminus U$.

{\em Case 2. For every $n\in A$, the set $\big(\{n\}\times  S\big)\cap U$ is closed in $\{n\}\times  S$ and $\w\SM A$ is infinite.} Taking into account that $U$ is non-closed and  $\{n\}\times  S$ is closed in $X$, the set $A$ must be also infinite. Moreover, since $S$ is pass-connected and $U$ is open, it follows that $U\cap\big(\w\times S\big)=A\times S$. Since $A$ and $\w\SM A$ are infinite, it follows that $(\w,\mathbf{0})\in X\setminus U$. It remains to note that the sequence $\{(n,\mathbf{0})\}_{n\in A} \subseteq A\times S \subseteq U$ converges to $(\w,\mathbf{0})$.
\smallskip

{\em Case 3. For every $n\in A$, the set $\big(\{n\}\times  S\big)\cap U$ is closed in $\{n\}\times  S$ and $\w\SM A$ is finite.} Since $S$ is pass-connected  and $U$ is open, the closeness of all $\big(\{n\}\times  S\big)\cap U$ implies that $U\cap\big(\w\times S\big)=A\times S$. Set
\[
F:=\big\{x\in \{\mathbf{0}\}\cup S_1: (\w,x)\notin U\big\},
\]
so that $F$ is a closed subset of $\{\mathbf{0}\}\cup S_1$. Since $\overline{U}\SM U\not=\emptyset$ and $\w\SM A$ is finite, we have $F\not=\emptyset$ (otherwise, $U=\big(A\times S\big) \cup \big( \{\w\}\times  (\{\mathbf{0}\}\cup S_1)\big)$ is closed in $X$). If $\mathbf{0}\in F$, then the sequence $\{(n,\mathbf{0})\}_{n\in A} \subseteq A\times S \subseteq U$ converges to $(\w,\mathbf{0})\in X\SM U$.

Assume that $\mathbf{0}\not\in F$. Then $F\subseteq S_1$. Since $S_1$ is dense in $\{\mathbf{0}\}\cup S_1$ and $F$ is a closed subset of $\{\mathbf{0}\}\cup S_1$, we have $S_1\setminus F\neq\emptyset$. As above, since $S_1$ is a pass-connected Fr\'{e}chet--Urysohn space (Exercise 3.10.D of \cite{Eng}), there exists a sequence $(x_n)_n\subseteq S_1\setminus F$ converging to some point $x\in F$. Then the sequence $\{(\w,x_n)\}_{n\in\w}\subseteq U$ converges to the point $(\w,x)\in X\setminus U$.

Therefore, in all three cases there is a sequence in $U$ which converges to some point in $\overline{U}\SM U$. Thus $X$ is $\kappa$-sequential.
\smallskip

{\em Claim 3. The space $X$ is not sequentially Ascoli.} Let $W$ be a nonempty open set in $S$ such that $\mathbf{0}\notin \overline{W}$. For every $n\in\w$, set $U_n:=\{n\}\times W$. Let $K$ be a compact subset of $X$. Since $X$ is pseudocompact (Claim 1), to show that $X$ is not sequentially Ascoli, by Theorem 1.2 of \cite{Gabr-pseudo-Ascoli}, we show that the set $I:=\{n\in\w: K\cap U_n\not=\emptyset\}$ is finite. Suppose for a contradiction that $I$ is infinite. For every $n\in I$, choose $x_n\in W$ such that $(n,x_n)\in K\cap U_n$. Then each cluster point of the sequence $\{(n,x_n)\}_{n\in I}$ has the form $(\w,x)$ with $x\in \overline{W}\cap S$ (recall that $S$ is $\w$-bounded). Since $S\cap S_1=\emptyset$ and $\mathbf{0}\notin \overline{W}$,  we have  $(\w,x)\notin X$ that is impossible. This contradiction finishes the proof of Claim 3.\qed
\end{proof}

Let us remark the following assertion.

\begin{proposition} \label{p:1-non-isol-kappa-sequential}
A $\kappa$-sequential space  $X$ with only one non-isolated point is Fr\'{e}chet--Urysohn.
\end{proposition}

\begin{proof}
Let $x_\infty$ be the unique non-isolated point of $X$. For every non-closed subset $A$ of $X$ we have $\overline{A}\SM A=\{x_\infty\}$ (otherwise, we have $x_\infty\in A$, and then $X\SM A$ is open being a union of isolated points in $X$ and hence $A$ is closed, a contradiction). As  $x_\infty\not\in A$, $A$ is open as a union of isolated points. Since $X$ is $\kappa$-sequential, there is a sequence in $A$ converging to a point $y\in X\SM A$. Therefore $y$ is non-isolated, and hence $y=x_\infty$. Thus $X$ is a  Fr\'{e}chet--Urysohn space.\qed
\end{proof}

%In the proof of Example \ref{exa:woca-non-Ascoli} and Proposition \ref{p:pass-connected-k-sequential} we essentially used the pass-connectedness of (sub)spaces. For zero-dimensional pseudocompact spaces we have the following statement.

By definition, any $\kappa$-sequential non-discrete space contains non-trivial convergent sequences.
\begin{example} \label{exa:bN-kappa-seq}
Each compact space without non-trivial convergent sequences {\rm(}for example,  $\beta\NN${\rm)} is open-compact attainable which is  not a $\kappa$-sequential space.
\end{example}
But for the compact {\em groups} the situation changes considerably.

\begin{proposition} \label{p:compact-group-kFU}
Each compact group $G$ is $\kappa$-Fr\'{e}chet--Urysohn.
\end{proposition}

\begin{proof}
By Ivanovskij--Kuz'minov's Theorem, $G$ is a dyadic space. Thus, by Corollary 3.5 of \cite{LiL}, $G$ is a  $\kappa$-Fr\'{e}chet--Urysohn space.\qed
\end{proof}
We will significantly strengthen Proposition \ref{p:compact-group-kFU} in Theorem \ref{t:feathered-group-kFU} below.

%\begin{proposition} \label{p:locompact-group-kFU}
%Each locally compact group $G$ is $\kappa$-Fr\'{e}chet--Urysohn.
%\end{proposition}
%\begin{proof}
%It suffices to verify that some neighborhood of the identity is $\kappa$-Fr\'{e}chet--Urysohn. Let $K$ be a compact neighborhood of the identity of type $G_\de$. Then $K$ is a Dugundji compact space \cite[Corollary 10.3.9]{ArT}. Consequently, $K$ is a dyadic compact space \cite[Theorem 10.1.3]{ArT}. Thus, by Corollary 3.5 of \cite{LiL}, $K$ is a $\kappa$-Fr\'{e}chet--Urysohn space.
%\qed\end{proof}

%The next remark  finishes the discussion on relationships between the notions from the diagram.

%\begin{remark} \label{rem:woca-relation} {\em
%(i) By \cite{GGKZ-2},  $C_p(\w_1)$ is $\kappa$-Fr\'{e}chet--Urysohn which is not a $k_\IR$-space. %In particular, there exist $\kappa$-sequential spaces which are not $s_\IR$-spaces.

%(ii) Clearly, any compact space is an open-compact attainable space. Therefore, $\beta\NN$ is open-compact attainable which is  not a $\kappa$-sequential space (because $\beta \NN$ has no non-trivial convergent sequences).}
%\end{remark}

We finish this section with several open problems.

\begin{problem}
Let $X$ be a countable spaces/group.

{\rm(i)} Can $X$ be $\kappa$-sequential but not sequential?

{\rm(ii)} Can $X$ be sequentially Ascoli but not Ascoli?
\end{problem}

Proposition \ref{p:zpwoca-dim-Ascoli} and Proposition \ref{p:zpwoca-Ascoli} motivate the next problem.
\begin{problem}
Let $X$ be a zero-dimensional  space. Is it true that if $X$ is weakly open-compact attainable, then $X$ is a sequentially Ascoli space?
\end{problem}

%In Question 4.1 of \cite{LiL}, Liu and Ludwig ask whether {\em a sequential, compact, $\kappa$-Fr\'{e}chet--Urysohn space is a Fr\'{e}chet--Urysohn space}. We do not know an answer to the following weaker problem.
%\begin{problem}
%Does there exist a sequential, $\kappa$-Fr\'{e}chet--Urysohn space which is not Fr\'{e}chet--Urysohn?
%\end{problem}

%It is known that a compact scattered space $K$ of countable scattered height $\mathsf{sh}(K)$ is sequential. The next questions are of independent interest.

%\begin{problem}
%Let $K$ be a (countable) compact scattered spaces (of countable scattered height).

%{\rm(i)} When is $K$ $\kappa$-Fr\'{e}chet--Urysohn?

%{\rm(ii)} Assume that $K$ is $\kappa$-Fr\'{e}chet--Urysohn and $\mathsf{sh}(K)<\w_1$. When is $K$ Fr\'{e}chet--Urysohn? Assuming that $\mathsf{sh}(K)$ is finite, is $K$ Fr\'{e}chet--Urysohn?
%\end{problem}

%\begin{problem}
%Characterize  the least ordinal $\alpha$ such that every sequential space $X$ of sequential order $so(X)<\alpha$ is $\kappa$-Fr\'{e}chet--Urysohn.
%\end{problem}
%In \cite[p.751]{Simon} Simon constructed two compact Fr\'{e}chet--Urysohn spaces $X$ and $Y$ whose square $X\times Y$ is not Fr\'{e}chet--Urysohn. Hence, by Theorem 4.2 of \cite{LiL}, $X\times Y$ is not $\kappa$-Fr\'{e}chet--Urysohn. As it is noticed in \cite[p.~400]{LiL}, $X\times Y$ is a sequential space. Therefore $\alpha$ is well-defined and $2\leq \alpha\leq\w_1$.

\begin{remark} {\em
Question 4.2 of \cite{LiL} asks whether each $\kappa$-metrizable space $\kappa$-Fr\'{e}chet--Urysohn. $\kappa$-metrizable spaces were considered by Shchepin \cite{Shchepin}. Since, by Corollary 2 of \cite{Shchepin}, $C_p(X)$ is $\kappa$-metrizable for every space $X$, Theorem \ref{t:Cp-kappa-FU} implies that the answer to this question is negative.}
\end{remark}

%Finally, we note that, by \cite{GGKZ-2},  $C_p(\w_1)$ is $\kappa$-Fr\'{e}chet--Urysohn which is not a $k_\IR$-space.

%%%%%%%%%%%%%%%%%%%%%%%%%%%%%%%
%%%%%%%%%%%%%%%%%%%%%%%%%%%%%%%
%%%%%%%%%%%%%%%%%%%%%%%%%%%%%%%
%%%%%%%%%%%%%%%%%%%%%%%%%%%%%%%
%%%%%%%%%%%%%%%%%%%%%%%%%%%%%%%

\section{Permanent properties} \label{sec:permanent}

%%%%%%%%%%%%%%%%%%%%%%%%%%%%%%%
%%%%%%%%%%%%%%%%%%%%%%%%%%%%%%%
%%%%%%%%%%%%%%%%%%%%%%%%%%%%%%%
%%%%%%%%%%%%%%%%%%%%%%%%%%%%%%%
%%%%%%%%%%%%%%%%%%%%%%%%%%%%%%%

It is known that the product of two spaces which have one of the properties from Definition \ref{def:compact-type} usually does not have the same property. It turnes out that the same holds also for the properties introduced in Definition \ref{def:oca}. %In Proposition 2.14 of \cite{Gabr-seq-Ascoli} it is proved that the square $V(\w_1)\times V(\w_1)$ is not an Ascoli space. Below we show that this space is not weakly open-compact attainable.

\begin{example} \label{exa:square-woca}
The space $V(\w_1)\times V(\w_1)$ is not weakly open-compact attainable.
\end{example}

\begin{proof}
By Theorem 20.2 of \cite{JW-1997}, there are two families $\{ A_\alpha: \alpha\in\w_1\}$ and $\{ B_\alpha: \alpha\in\w_1\}$ of infinite subsets of $\w$ satisfying the following two conditions:
\begin{enumerate}
\item[{\rm (a)}] $A_\alpha \cap B_\beta$ is finite for all $\alpha,\beta \in\w_1$,
\item[{\rm (b)}] for no $D\subseteq \w$, all sets $A_\alpha\SM D$ and $B_\alpha\cap D$, $\alpha\in\w_1$, are finite.
\end{enumerate}

Set $z:= (x_\infty,x_\infty)$ and
\[
U:=\big\{ (x_{n,\alpha},x_{n,\beta}) \in V(\w_1)\times V(\w_1): \; \alpha,\beta \in\w_1\, \mbox{ and } \, n\in A_\alpha \cap B_\beta\big\}.
\]
Since the points $ (x_{n,\alpha},x_{n,\beta})$ are isolated, the set $U$  is open. In \cite{GT-1982}, Gruenhage and Tanaka proved that $z$ is an accumulation point of $U$. Thus $\overline{U}\SM U\not=\emptyset$.

To show that $V(\w_1)\times V(\w_1)$ is not weakly open-compact attainable, by Proposition \ref{p:non-woca}, it suffices to prove that $U\cap K$ is finite for every compact subset of  $V(\w_1)\times V(\w_1)$. To this end, let $K$ be a compact subset of $V(\w_1)\times V(\w_1)$. Since the projection of $K$ onto each coordinate is a compact subset of $V(\w_1)$, without loss of generality we can assume that $K$ has the form
\[
K=\big\{ (x_{n,\alpha},x_{n,\beta})\in V(\w_1)\times V(\w_1): \; \alpha,\beta\in \Gamma \, \mbox{ and }\, n\in\w\big\}\cup\{z\},
\]
where $\Gamma$ is a finite subset of $\w_1$. Then, by (a) and the finiteness of $\Gamma$, we obtain that
\[
U \cap K\subseteq \big\{ (x_{n,\alpha},x_{n,\beta})\in V(\w_1)\times V(\w_1): \; \alpha,\beta\in \Gamma \, \mbox{ and } \, n\in A_\alpha \cap B_\beta\big\}
\]
is finite, as desired.\qed
\end{proof}

Below we consider another example. Recall that, by  (v) of Example \ref{exa:S2-lambda}, $S_2(\lambda)$ is a $k$-space.

\begin{example} \label{exa:S2-lambda-prod}
If $\lambda\geq\w_1$, then  $S_2(\lambda)\times S_2(\lambda)$  is neither an Ascoli space nor  a weakly open-compact attainable space.
\end{example}

\begin{proof}
We claim that there is a quotient map $Q$ from $S_2(\lambda)\times S_2(\lambda)$ onto $V(\w_1)\times V(\w_1)$. Indeed, by (ii) and Theorem 3.7.7 of \cite{Eng}, $f\times f$ is a perfect map. Observe also that perfect mappings are closed (Theorem 3.7.13 of \cite{Eng}) and hence quotients (Corollary 2.4.8 of \cite{Eng}). On the other hand, the map $g:V(\lambda)\times V(\lambda)\to V(\w_1)\times V(\w_1)$ defined by
\[
g(x,y)=(x,y) \mbox{ if } (x,y)\in V(\w_1)\times V(\w_1), \mbox{ and } g(x,y)=(x_\infty,x_\infty) \mbox{ otherwise},
\]
is continuous, and hence, by Corollary 2.4.6 of \cite{Eng}, is a quotient map. Therefore, by Corollary 2.4.4 of \cite{Eng}, the map $Q:=g\circ (f\times f)$ is a quotient map.
\smallskip

Suppose for a contradiction that $S_2(\lambda)\times S_2(\lambda)$ is an Ascoli  space. Then, by Proposition 3.3 of \cite{GR-prod} which states that $R$-quotient images of Ascoli spaces are Ascoli, we obtain that $V(\w_1)\times V(\w_1)$ is an Ascoli space. But this contradicts Proposition 2.14 of \cite{Gabr-seq-Ascoli}. Thus $S_2(\lambda)\times S_2(\lambda)$ is not Ascoli.
\smallskip

Suppose for a contradiction that $S_2(\lambda)\times S_2(\lambda)$ is a weakly open-compact attainable space. Then, by Lemma \ref{l:exa-dagger-map}(i) and Proposition \ref{p:quotient-woca}(i), we obtain that $V(\w_1)\times V(\w_1)$ is weakly open-compact attainable. However, this contradicts Example \ref{exa:square-woca}. Thus $S_2(\lambda)\times S_2(\lambda)$ is not a weakly open-compact attainable space. \qed
\end{proof}

\begin{problem}
Let $n\geq 2$. Does there exist a  weakly open-compact attainable (resp., $\kappa$-sequential, Ascoli or sequentially Ascoli) space $X$ such that $X^n$ has the same property but $X^{n+1}$ not?
\end{problem}

\begin{problem}
Does there exist a  weakly open-compact attainable (resp., $\kappa$-sequential, Ascoli or sequentially Ascoli) space $X$ such that $X^n$ has the same property for every $n\in\NN$, but $X^\w$ not?
\end{problem}

\begin{problem}
Is it true that $S_2(\lambda)^n$ is strongly zero-dimensional for every $n\in\NN$?
\end{problem}

In fact, we know only a unique class of spaces which are sequentially Ascoli but not Ascoli, this is the class of non-discrete $P$-spaces. Even a partial positive answer to the next problem will give a totally different type of such spaces.
\begin{problem}
Is it true that $S_2(\lambda)^n$ is a sequentially Ascoli space for every $n\in\NN$? What about $S_2(\lambda)\times S_2(\lambda)$?
\end{problem}

In fact,  Theorem 1 of \cite{Bagley-Weddington-69} states that if $X$ is a non-discrete $T_1$-space, then  there exists  an infinite cardinal $\lambda$ such that the product $V(\lambda)\times X$ is not a $k'$-space. For spaces $X$ with better separation axioms, this statement can be improved.

\begin{theorem} \label{t:Veer-non-discrete}
For every non-discrete space $X$ and each infinite cardinal $\lambda$, the product $V(\lambda)\times X$ is not open-compact attainable.
\end{theorem}

\begin{proof}
We shall consider $\w$ as a subset of $\lambda$.
If $X$ is a $P$-space, then $X$ is not an Ascoli space by Proposition 2.9 of \cite{Gabr:weak-bar-L(X)}. Therefore, by Proposition \ref{p:quot-oca} proved below, the product  $V(\lambda)\times X$ is not open-compact attainable. So we assume that $X$ is not a $P$-space.

By Proposition 3.14 of \cite{Gab-Pel}, there is a point $z\in X$ and a sequence $\{U_m\}_{m\in\w}$ of open subsets of $X$ such that $U_n\cap U_m=\emptyset$ for all distinct $n,m\in \w$ and such that
\begin{equation} \label{equ:Veer-non-discrete}
z\in \overline{\bigcup_{m\in\w} U_m} \setminus \bigcup_{m\in\w} U_m \;\;\mbox{ and }\;\; z \not\in \overline{\bigcup_{m\leq n} U_m} \;\; (n\in \w).
\end{equation}
Consider the following subset of $V(\lambda)\times X$
\[
W:=\bigcup_{n,m\in\w} \big\{x_{n+m,m}\big\}\times U_m.
\]
Since all the points $x_{n+m,m}$ are isolated in $V(\lambda)$, the set $W$ is open.

We claim that $(x_\infty,z)\in \overline{W}\SM W$. Indeed, it is clear that $(x_\infty,z)\not\in W$. To show that $(x_\infty,z)\in \overline{W}$, take an arbitrary standard neighborhood
\[
\mathcal{O}:=\{ x_{n,\alpha}: \alpha\in\lambda \mbox{ and } n\geq k_\alpha\}\cup\{ x_\infty\}
\]
of $x_\infty$, where $k_\alpha\in \w$ for every $ \alpha\in\lambda$, and choose a neighborhood $U$ of $z$. By (\ref{equ:Veer-non-discrete}), take $m\in\w$ such that $U_m\cap U\not=\emptyset$, and set $n=k_m$. Then $(\mathcal{O}\times U)\cap \big(\big\{x_{n+m,m}\big\}\times U_m\big) \not=\emptyset$. Thus $(x_\infty,z)\in \overline{W}$.

To show that $V(\lambda)\times X$ is not open-compact attainable, we fix an arbitrary compact subset  $K$ of $V(\lambda)\times X$ and show that $(x_\infty,z)\not\in \overline{K\cap W}$. To this end,  we observe that there is a finite subset $F$ of $\lambda$ such that the projection of $K$ onto $V(\lambda)$ is contained in the compact set
\[
K_F :=\big\{ x_{n,\alpha}: n\in\w \mbox{ and } \alpha\in F\big\}\cup\{ x_\infty\}.
\]
Therefore
\[
K\cap W \subseteq K_F\cap W \subseteq V(\lambda)\times \bigcup_{m\in F\cap \w} U_m,
\]
and hence, by (\ref{equ:Veer-non-discrete}), we have  $(x_\infty,z)\not\in \overline{K\cap W}$. Thus $V(\lambda)\times X$ is not open-compact attainable.\qed
\end{proof}

Let $X$ and $Y$ be spaces with only one non-isolated point. Proposition 4.1 of \cite{LiL} states that $X\times Y$ is $\kappa$-Fr\'{e}chet--Urysohn if, and only, if $X\times Y$ is Fr\'{e}chet--Urysohn. We know (Proposition \ref{p:1-non-isol-kappa-sequential}) that if $X$ is $\kappa$-sequential, then it is Fr\'{e}chet--Urysohn. In the light of these results and Example \ref{exa:square-woca}, one can natural ask whether the $\kappa$-sequentiality of $X\times Y$ implies that $X\times Y$ is Fr\'{e}chet--Urysohn. We answer this problem in the negative in the next example which has also the following another meaning. Corollary 4.1 of \cite{LiL} states that there are two compact,  Fr\'{e}chet--Urysohn spaces, such that $X\times Y$ is not $\kappa$-Fr\'{e}chet--Urysohn.
Analogously, the product of an open-compact attainable space and a compact space can be not open-compact attainable as the following example shows.

\begin{example} \label{exa:product-V-S-non-oca}
Let $V(\w)$ be the countable Fr\'{e}chet--Urysohn fan. Then the product $V(\w)\times (\w+1)$ is a sequential space which is not open-compact attainable.
\end{example}

\begin{proof}
The product $V(\w)\times (\w+1)$ is a sequential space by Exercise 3.10.I(b) of \cite{Eng}, and it is  not open-compact attainable by Theorem \ref{t:Veer-non-discrete}.\qed
\end{proof}

On the other hand, the product of a $k_\IR$-space and a compact space is always a $k_\IR$-space by Theorem 4.3 of \cite{GR-prod}. Analogously, by Theorem 4.10 of \cite{GR-prod}, the product of an $s_\IR$-space and a compact sequential space is always an $s_\IR$-space. It is also known (see Theorem 4.2 of \cite{Gabr-seq-Ascoli}) that the product of an Ascoli space and a compact space is Ascoli. These results motivate the following problem.

\begin{problem}
Let $X$ be a weakly open-compact attainable space (resp., a $\kappa$-sequential space), and let $K$ be a compact space (resp., a $\kappa$-sequential compact space). Is $X\times K$  weakly open-compact attainable  (resp., $\kappa$-sequential)?
\end{problem}

Observe that the product $V(\w)\times (\w+1)$, being sequential, is an Ascoli space. Therefore one cannot strengthen Theorem \ref{t:Veer-non-discrete} replacing open-compact attainability by the property of being an Ascoli space. On the other hand, we can ask whether there is a ``nice'' space $X$ such that the product $V(\w)\times X$ is not sequentially Ascoli. Taking into account Example \ref{exa:square-woca}, by ''nice'' it is naturally to understand metric spaces. We answer this problem in the affirmative in the next example in which $\mathbb{Q}$ denotes the space of rational numbers with its usual topology. Observe also that the space $V(\w)\times \mathbb{Q}$ is not a $k$-space by results of Michael \cite{Michael-68}, and hence  our example essentially generalizes this result.

\begin{example} \label{exa:Veer-Q}
The product $V(\w)\times \mathbb{Q}$ is neither weakly open-compact attainable nor sequentially Ascoli.
\end{example}
%In fact, Michael \cite{Michael-68} proved that $V(\w)\times \mathbb{Q}$ is not a $k$-space. So Example \ref{exa:Veer-Q} essentially generalizes this result.

\begin{proof}
Since $X:=V(\w)\times \mathbb{Q}$ is countable, by (i) of Corollary \ref{c:counable-X-woca-Seq-Ascoli}, it suffices to prove that $X$ is not $\kappa$-sequential.

Choose two sequences $\{a_k\}_{k\in\w}$ and  $\{b_k\}_{k\in\w}$ of irrational numbers such that
\begin{enumerate}
\item[(a)] $0<b_{k+1}<a_k<b_k$ for every $k\in\w$;
\item[(b)] $b_k\to 0$ as $k\to\infty$.
\end{enumerate}
and set  $W_k:=(a_k,b_k)\cap \mathbb{Q}$. Then $W_k$ is a clopen subset of $\mathbb{Q}$.

For every $k\in\w$, choose two sequences $\{a_{n,k}\}_{n\in\w}$ and  $\{b_{n,k}\}_{n\in\w}$ of irrational numbers such that
\begin{enumerate}
\item[(c)] $a_k<b_{n+1,k}<a_{n,k}<b_{n,k}< b_k$ for every $n\in\w$;
\item[(d)] $b_{n,k}\to a_k$ as $n\to\infty$,
\end{enumerate}
and set $W_{n,k}:=(a_{n,k},b_{n,k})\cap \mathbb{Q}$. Then for every $n,k\in\w$, $W_{n,k}$ and $\bigcup_{n\in\w} W_{n,k}$ are clopen subsets of $\mathbb{Q}$. It is easy to see that if $0\not=x\in\mathbb{Q}$, then $x$ has a neighborhood which intersects with at most one of the sets $W_{i,j}$, $i,j\in\w$.

For every $n,k\in\w$, we set
\[
U_{n,k}:=\{x_{n,k}\}\times W_{n,k} \;\; \mbox{ and }\;\;  U:=\bigcup_{n,k\in\w} U_{n,k}.
\]
It is clear that all $U_{n,k}$ are clopen subsets of $X$, and $U$ is an open subset of $X$.
\smallskip

{\em Claim 1. $\overline{U}\SM U=\{(x_\infty,0)\}$.}
\smallskip

To prove the claim, let $(y,z)\in \overline{U}\SM U$. If $y=x_{n,k}$ for some $n,k\in\w$, then $y=x_{n,k}$ is isolated in $V(\w)$. Therefore $(\{x_{n,k}\}\times \mathbb{Q})\cap U= \{x_{n,k}\}\times W_{n,k}$ is a closed subset of $U$. Hence $(y,x)\in U$, a contradiction. Thus $y=x_\infty$.

Assume that $z\not= 0$. Choose a neighborhood $\mathcal{O}$ of $z$ which intersects with at most one of the sets $W_{i,j}$, say  $W_{n,k}$. Choose a neighborhood $V$ of $x_\infty$ such that $x_{n,k}\not\in V$. Then $(V\times \mathcal{O})\cap U=\emptyset$, and hence $(x_\infty,z)\not\in \overline{U}$, a contradiction. Thus $z=0$.

It remains to show that $(x_\infty,0)\in \overline{U}\SM U$. Note first that, by construction,  $(x_\infty,0)\not\in U$. We prove that  $(x_\infty,z)\in \overline{U}$. Let $P$ be a neighborhood of $x_\infty$ in $V(\w)$, and $R=(-r,r)\cap \mathbb{Q}$ be a neighborhood of $0$ in $\mathbb{Q}$. Choose $k\in\w$  such that $b_k<r$, so that $W_{n,k}\subseteq R$ for every $n\in\w$. For the chosen $k$, take $n\in\w$ such that $x_{n,k}\in P$. Then $U_{n,k}\subseteq P\times R$. Thus  $(x_\infty,0)\in \overline{U}$, and Claim 1 is proved.
\smallskip

To finish the proof, by Claim 1, it suffices to show that there is no sequences in $U$ converging to $(x_\infty,0)$. Suppose for a contradiction that there is a sequence $\big\{ \big(x_{n_i,k_i},z_{n_i,k_i}\big)\big\}_{i\in\w}$ in $U$ which converges to  $(x_\infty,0)$. Then $x_{n_i,k_i}\to x_\infty$. Therefore there is $k\in\w$ such that $k_i\leq k$ for every $i\in\w$. So, without loss of generality we can assume that $k_i=k$ and $x_{n_i,k}\to x_\infty$. This implies that $z_{n_i,k}\in W_{n_i,k} \subseteq W_k$, and hence  $z_{n_i,k}\not\to 0$ as $i\to\infty$. Therefore $\big(x_{n_i,k_i},z_{n_i,k_i}\big)\not\to (x_\infty,0)$, a contradiction. \qed
\end{proof}

In \cite[Example 10.5]{Michael1972} Michael constructed  a separable metrizable space $X$ and a compact space $Y$ such that $X\times Y$ is not $k'$-space. This result and Examples \ref{exa:product-V-S-non-oca} and \ref{exa:Veer-Q} motivate the following problem.

\begin{problem}
Let $X$ be a separable metrizable space (for example, $X=\mathbb{Q}$), and let $Y$ be a compact space. Is it true that the product $X\times Y$ is open-compact attainable?
\end{problem}

Now we consider the case of products of (two) pseudocompact spaces. Proposition 2.7 of \cite{Gabr-pseudo-Ascoli} states that the product of an arbitrary family of Ascoi {\em pseudocompact} spaces is Ascoli. This proposition and Theorem 4.6 of \cite{GR-prod} imply that the product of two pseudocompact  $k_\IR$-spaces is a pseudocompact $k_\IR$-space. Analogously, by Proposition 2.7 of \cite{Gabr-pseudo-Ascoli} and Theorem 4.9 of \cite{GR-prod},   the product of two pseudocompact  $s_\IR$-spaces is a pseudocompact $s_\IR$-space. These results motivate the following problem.
\begin{problem}
Let $X$ and $Y$ be  two $\kappa$-sequential (resp., open-compact attainable or weakly  open-compact attainable) pseudocompact spaces. Is it true that $X\times Y$ has the same property? What about $X\times X$ for $X$ from Example \ref{exa:woca-non-Ascoli}?
\end{problem}

%\begin{example} \label{exa:prod-k-seq-non-FU}
%Let $V(\w)$ be the countable Fr\'{e}chet--Urysohn fan. Then the product $V(\w)\times (\w+1)$ is a sequential non-Fr\'{e}chet--Urysohn space.
%\end{example}

%\begin{proof}
%The product $V(\w)\times (\w+1)$ is a sequential space by Exercise 3.10.I(b) of \cite{Eng}, and, by Example 2.2 of \cite{Gruenhage-1}, it is not Fr\'{e}chet--Urysohn.\qed
%\end{proof}

In Example \ref{exa:square-woca} the space $V(\w_1)$ is not first-countable. For first-countable spaces the situation changes.
The following was observed by Mr\'{o}wka \cite[Theorem~3.1]{Mrowka}  even though the notion of $\kappa$-Fr\'{e}chet--Urysohn was not yet in use; it is explicitly formulated in Fact 2.1 of \cite{LiL}.

\begin{fact} \label{f:product-kFU}
Let $X=\prod_{\alpha\in \AAA}X_\alpha$ be a product of first countable spaces. Then $X$ is a  $\kappa$-Fr\'{e}chet--Urysohn space.
\end{fact}

\begin{remark}{\em
Since $\kappa$-Fr\'{e}chet--Urysohn spaces are  open-compact attainable and $\kappa$-sequential, it follows that any product of first countable spaces is open-compact attainable and $\kappa$-sequential. Example \ref{exa:square-woca}  shows that it is impossible to weaken the first countability to Fr\'{e}chet--Urysohness.}
\end{remark}

In general, the properties of being open-compact attainable, $\kappa$-sequential or weakly open-compact attainable are not {\em closed hereditary}. Indeed, let $X=\NN\cup\{x_\ast\}$, where $x_\ast\in\beta\NN\SM\NN$. Since $X$ does not contain infinite compact subsets and $\NN$ is open in $X$, Proposition \ref{p:non-woca} implies that $X$ is not weakly open-compact attainable. Being countable and hence realcompact, $X$ is closely embedded into some power $\IR^\lambda$ which is even a  $\kappa$-Fr\'{e}chet--Urysohn space  by Fact \ref{f:product-kFU}.

Arhangel'skii proved (see \cite[3.12.15]{Eng}) that if $X$ is a hereditary $k$-space (i.e., {\em every} subspace of $X$ is a $k$-space), then $X$ is Fr\'{e}chet--Urysohn. In Theorem 2.21 of \cite{Gabr-seq-Ascoli} this result was generalized as follows: a hereditary Ascoli space is Fr\'{e}chet--Urysohn. In the next theorem we strengthen this remarkable Arhangel'skii theorem in another direction.

\begin{theorem} \label{t:woca-hereditary}
A hereditary weakly open-compact attainable space $X$ is Fr\'{e}chet--Urysohn.
\end{theorem}

\begin{proof}
Taking into account Arhangel'skii's theorem \cite[3.12.15]{Eng} it suffices to prove that $X$ is a $k$-space. Assuming the converse we could find a non-closed subspace $A$ of $X$ such that $A\cap K$ is closed in $K$ for every compact subspace $K$ of $X$. Choose an arbitrary point $z\in \overline{A}\SM A$ and consider the subspace $Y:= A\cup\{ z\}$ of $X$.

We claim that if $K$ is a compact subspace of $Y$, then either $z\not\in K$ or $z$ is an isolated point of $K$. Indeed, if $z$ is a non-isolated point of $K$, then $z \in \overline{A\cap K}$. Since $z\not\in A$, we obtain that $A\cap K$ is not closed in $K$ that contradicts the choice of $A$.

Clearly, $A$ is an open subset of $Y$. Since, by assumption of the theorem, $Y$ is  weakly open-compact attainable, there is a compact subset $K$ of $Y$ such that $z\in K\cap \overline{A\cap K}$. In particular, $z$ is  a non-isolated point of $K$ which contradicts the claim.\qed
\end{proof}

It is known (see Proposition 2.2 of \cite{Gabr-pseudo-Ascoli}), that if $S$ is a dense Ascoli pseudocompact subspace of a space $X$, then also $X$ is Ascoli. An analogous result is not valid for (weakly) open-compact attainable  pseudocompact spaces or  $\kappa$-sequential pseudocompact spaces. Indeed, let $X$ be the union of $S$ and $\{\mathbf{1}\}$ from Example \ref{exa:non-woca-Ascoli}. Then $S$ is a Fr\'{e}chet--Urysohn $\w$-bounded space. However, by Example \ref{exa:non-woca-Ascoli},  the space $X$ (which is the  closure of $S$)  does not belong to any of these classes of spaces.

%%%%%%%%%%%%%%%%%%%%%%%%%%%%%%%%%%%%%%%
%%%%%%%%%%%%%%%%%%%%%%%%%%%%%%%%%%%%%%%
%%%%%%%%%%%%%%%%%%%%%%%%%%%%%%%%%%%%%%%
%%%%%%%%%%%%%%%%%%%%%%%%%%%%%%%%%%%%%%%
%%%%%%%%%%%%%%%%%%%%%%%%%%%%%%%%%%%%%%%

\section{$\kappa$-pseudo open maps and weakly  $\kappa$-pseudo open maps} \label{sec:maps}

%%%%%%%%%%%%%%%%%%%%%%%%%%%%%%%%%%%%%%%
%%%%%%%%%%%%%%%%%%%%%%%%%%%%%%%%%%%%%%%
%%%%%%%%%%%%%%%%%%%%%%%%%%%%%%%%%%%%%%%
%%%%%%%%%%%%%%%%%%%%%%%%%%%%%%%%%%%%%%%
%%%%%%%%%%%%%%%%%%%%%%%%%%%%%%%%%%%%%%%

A continuous image of a (sequentially) Ascoli pseudocompact space is also Ascoli pseudocompact, see Proposition 2.1 of \cite{Gabr-pseudo-Ascoli}. However, an analogous result for (weakly) open-compact attainable pseudocompact  spaces and $\kappa$-sequential pseudocompact  spaces is not true in general. Indeed, let $Y$ be the disjoint union of $S$ and $\{\mathbf{1}\}$ from Example \ref{exa:non-woca-Ascoli}. Since $S$ is a Fr\'{e}chet--Urysohn $\w$-bounded space, the space $Y$ belongs to any of those spaces. However, by Example \ref{exa:non-woca-Ascoli}, the continuous image $X$ of $Y$ under the identity mapping is not weakly open-compact attainable. Therefore, to obtain positive results we should consider some subclasses of the class of continuous mappings. The first one is the class of  $\kappa$-pseudo open maps. Recall that $F:X\to Y$ is {\em  $\kappa$-pseudo open } if whenever $f^{-1}(y)\subseteq U$ where $y\in Y$ and $U$ is open in $X$, then $y\in \Int\big(\overline{f(U)}\big)$.

\begin{lemma} \label{l:charac-dagger-map}
Let $f:X\to Y$ be a continuous surjective map between spaces $X$ and $Y$. Then the following assertions are equivalent:
\begin{enumerate}
\item[{\rm(i)}] $f$ is a  $\kappa$-pseudo open map;
\item[{\rm(ii)}] $\overline{U}=f\big(\overline{f^{-1}(U)}\big)$ for any open $U\subseteq Y$;
\item[{\rm(iii)}] $f^{-1}(y)\cap \overline{f^{-1}(U)}\not=\emptyset$ for every open non-closed set $U\subseteq Y$ and each $y\in \overline{U}\SM U$.
\end{enumerate}
\end{lemma}

\begin{proof}
(i)$\Ra$(ii)  Assume that $f$ is a $\kappa$-pseudo open map. Let $U\subseteq Y$ be an open set. Set $F:=\overline{f^{-1}(U)}$. We have to verify that $\overline{U}=f(F)$. The continuity of $f$ implies  $f(F) \subseteq \overline{U}$.

To prove the inverse inclusion $\overline{U} \subseteq f(F)$, we suppose for a contradiction that there is $y\in \overline{U}$ such that $y\notin f(F)$. Then $f^{-1}(y)\subseteq W= X\setminus F$. Since $f$ is $\kappa$-pseudo open, we have $y\in \Int \cl{f(W)}$. Therefore, $f(W)\cap U\neq\emptyset$. Let $z\in f(W)\cap U$. Since $z\in U$, we have $f^{-1}(z)\subseteq f^{-1}(U)\subseteq F=X\setminus W$, and hence $z\notin f(W)$. A contradiction.
\smallskip

(ii)$\Ra$(i) Let $W\subseteq X$ be an open set, and let $y\in Y$ be such that $f^{-1}(y)\subseteq W$. We show that $y\in \Int{f(\overline{W})}$. Assuming the converse we obtain that $y\in \overline{U}$, where $U=Y\SM\overline{f(W)}$. Therefore, $y\in f\big(\overline{f^{-1}(U)}\big)$. Let $x\in f^{-1}(y)\cap \overline{f^{-1}(U)}$. Therefore the set $W$ is an open neighborhood of the point $x$ such that  $W\cap f^{-1}(U)=\emptyset$. Thus, $x\notin \overline{f^{-1}(U)}$, a contradiction.

(ii)$\Ra$(iii)  Let $U$ be an open non-closed set and let $y\in \overline{U}\SM U$. By (ii), we have $y\in f(\overline{f^{-1}(U)})$. Take an arbitrary $x\in \overline{f^{-1}(U)}$ such that $f(x)=y$. Then $x\in f^{-1}(y) \cap \overline{f^{-1}(U)}$. Therefore $ f^{-1}(y) \cap \overline{f^{-1}(U)}\neq\emptyset$.
\smallskip

(iii)$\Ra$(ii)  We show that $\overline{U}=f\big(\overline{f^{-1}(U)}\big)$ for any open $U\subseteq Y$. The inclusion $f\big(\overline{f^{-1}(U)}\big)\subseteq \overline{U}$ is clear. To prove the converse inclusion, let $y\in \overline{U}$. If $y\in U$, it is trivial that $y\in f\big(\overline{f^{-1}(U)}\big)$. If $y\in \overline{U}\SM U$, then $f^{-1}(y) \cap \overline{f^{-1}(U)}\neq\emptyset$. Take a point $x\in  f^{-1}(y) \cap \overline{f^{-1}(U)}$. Then $y=f(x)\in f\big(\overline{f^{-1}(U)}\big)$. Thus $\overline{U}\subseteq f\big(\overline{f^{-1}(U)}\big)$, and hence $\overline{U}=f\big(\overline{f^{-1}(U)}\big)$.\qed
\end{proof}

The following class will be used also to characterize weakly open-compact attainable spaces and $\kappa$-sequential spaces in the next section.

\begin{definition} \label{def:func-woca} {\em
Let $X$ and $Y$ be spaces. A surjective continuous map $f:X\to Y$ is called a {\em weakly $\kappa$-pseudo open map} if  an open set $U\subseteq Y$ is closed if and only if $f^{-1}(U)$ is closed.}
\end{definition}

Note that if $X$ is connected, then any surjective continuous map $f:X\to Y$ is a weakly  $\kappa$-pseudo open map.

It will be convenient to use the following characterizations of weakly $\kappa$-pseudo open maps.

\begin{lemma} \label{l:charac-dagger-map-weak}
Let $f:X\to Y$ be a continuous surjective map between spaces $X$ and $Y$. Then $f$ is a weakly $\kappa$-pseudo open map if, and only if, for every open non-closed set $U\subseteq Y$ there is  $y\in \overline{U}\SM U$ such that $f^{-1}(y)\cap \overline{f^{-1}(U)}\not=\emptyset$.
Consequently, each $\kappa$-pseudo open map is a weakly $\kappa$-pseudo open map.
\end{lemma}

\begin{proof}
$(\Ra)$ Assume that $f$ is a weakly $\kappa$-pseudo open map, and let  $U\subseteq Y$ be an open non-closed set. Then the open set $f^{-1}(U)$ is also not closed.  Take $x\in \overline{f^{-1}(U)}\SM f^{-1}(U)$ and set $y:=f(x)$. Then $y\not\in U$, $y\in f\big(\overline{f^{-1}(U)}\big)\subseteq \overline{U}$ and $x\in f^{-1}(y)\cap \overline{f^{-1}(U)}$.

$(\Leftarrow)$ Let $U\subseteq Y$ be open. If $U$ is closed, then clearly $f^{-1}(U)$ is also closed. Conversely, assume that $f^{-1}(U)$ is closed and suppose for a contradiction that $U$ is not closed. Then, by assumption, there is $y\in \overline{U}\SM U$ such that $f^{-1}(y)\cap \overline{f^{-1}(U)}\not=\emptyset$. Since $f^{-1}(U)$ is closed, it follows that  $f^{-1}(y)\cap f^{-1}(U)\not=\emptyset$ and hence $y\in f\big(f^{-1}(U)\big)=U$, a contradiction. Thus $U$ must be closed.
\smallskip

The last assertion follows from Lemma \ref{l:charac-dagger-map}.\qed
\end{proof}

Proposition 3.2 of \cite{LiL} states that the composition of a pseudo-open mapping and a $\kappa$-pseudo-open mapping is $\kappa$-pseudo-open. Below we strengthen this result.

\begin{proposition} \label{p:composition-dagger}
Let $f:X\to Y$ and $g:Y\to Z$ be continuous maps. If $f$ and $g$ are (weakly) $\kappa$-pseudo open, then also $g\circ f$ is a  (weakly) $\kappa$-pseudo open map.
\end{proposition}

\begin{proof}
Let $U\subseteq Z$ be an open set. Set $V:=g^{-1}(U)$.

Assume that $f$ and $g$ are $\kappa$-pseudo open. We will use $(1)\LRa(2)$ in Lemma \ref{l:charac-dagger-map}. Since $g$ and $f$ are $\kappa$-pseudo open, $\overline{U}=g(\overline{V})$ and $\overline{V}=f\big(\overline{f^{-1}(V)}\big)$. Therefore, $\overline{U}=(g\circ f)\big(\overline{(g\circ f)^{-1}(U)}\big)$. Since $U$ was arbitrary it follows that $g\circ f$ is $\kappa$-pseudo open.

Assume that $f$ and $g$ are weakly $\kappa$-pseudo open. We shall use Lemma \ref{l:charac-dagger-map-weak}. Assume that $(g\circ f)^{-1}(U)$ is closed. Since $f$ is weakly $\kappa$-pseudo open and $f^{-1}(V)=(g\circ f)^{-1}(U)$ is closed, we obtain that $V$ is closed. As $g$ is weakly $\kappa$-pseudo open and $g^{-1}(U)=V$  is closed, $U$ is closed. Thus $g\circ f$ is weakly $\kappa$-pseudo open map. \qed
\end{proof}

%Following Karnik and Willard \cite{Kar-Wil},
Recall that a continuous mapping $p:X\to Y$ is called {\em $R$-quotient} ({\em real-quotient}) if $p$ is continuous, onto, and every function $\phi:Y\to\IR$ is continuous whenever the composition $\phi\circ p$ is continuous. Clearly, every quotient mapping as well as each surjective open mapping is $R$-quotient.

\begin{lemma} \label{l:exa-dagger-map}
Let $f:X\to Y$  be a surjective continuous map between spaces $X$ and $Y$.
\begin{enumerate}
\item[{\rm(i)}] If $f$ is $R$-quotient, then $f$ is a weakly $\kappa$-pseudo open  map.
\item[{\rm(ii)}] If $f$ is pseudo open or open, then $f$ is a  $\kappa$-pseudo open map.
\end{enumerate}
\end{lemma}

\begin{proof}
(i) Let $U$ be an open non-closed subset of $Y$. We show that $W:=f^{-1}(U)$ is an open non-closed subset of $X$. Indeed, suppose for a contradiction that $W$ is also closed in $X$. Consider the characteristic function $\mathbf{1}_U$ of $U$ on $Y$. Clearly, $\mathbf{1}_U$ is discontinuous. Since $p$ is $R$-quotient, $\mathbf{1}_U\circ p$ is discontinuous as well. However, since $W$ is clopen, the function $\mathbf{1}_U\circ p$ is continuous, a contradiction. Thus $W$ is not closed. Thus $f$ is a weakly  $\kappa$-pseudo open map.

(ii) is obvious.\qed
\end{proof}

\begin{proposition}  \label{p:quot-oca}
Let $f:X\to Y$ be a \dag-map between spaces $X$ and $Y$.
\begin{enumerate}
\item[{\rm(i)}] If $X$ is open-compact attainable, then so is $Y$.
\item[{\rm(ii)}]  If $X$ is $\kappa$-Fr\'{e}chet--Urysohn, then so is $Y$.
\end{enumerate}
\end{proposition}

\begin{proof}
Let $U$ be an open non-closed subset of $Y$, and let $y\in \overline{U}\SM U$. Since $f$ is a  $\kappa$-pseudo open map, Lemma \ref{l:charac-dagger-map} implies that $f^{-1}(y)\cap \overline{f^{-1}(U)}\not=\emptyset$. Set $W:=f^{-1}(U)$. Choose a point $z\in f^{-1}(y)\cap \overline{W}$. Then  $f(z)=y$ and  $z\in\overline{W}\SM W$.
\smallskip

(i) Since $X$ is open-compact attainable, there is a subset $A$ of $W$ such that $z\in \overline{A}$ and $\overline{A}$ is a compact subset of $X$. It is clear that $K:=f(\overline{A})$ is a compact subset of $Y$ such that $y\in K\cap \overline{K\cap U}$. Thus $Y$ is an open-compact attainable space.\qed
\smallskip

(ii)  Since $X$ is  $\kappa$-Fr\'{e}chet--Urysohn, there is a sequence $\{x_n\}_{n\in\w}\subseteq W$ which converges to $z$. Then the sequence $\{f(x_n)\}_{n\in\w}\subseteq U$ converges to $y$. Thus $Y$ is a  $\kappa$-Fr\'{e}chet--Urysohn space.\qed
\end{proof}

\begin{proposition}  \label{p:quotient-woca}
Let $f:X\to Y$ be a weakly  $\kappa$-pseudo open map between spaces $X$ and $Y$.
\begin{enumerate}
\item[{\rm(i)}] If $X$ is weakly open-compact attainable, then so is $Y$.
\item[{\rm(ii)}]  If $X$ is $\kappa$-sequential, then so is $Y$.
\end{enumerate}
\end{proposition}

\begin{proof}
Let $U$ be an open non-closed subset of $Y$. Set $W:=f^{-1}(U)$. Since $f$ is a weakly  $\kappa$-pseudo open map, we apply Lemma \ref{l:charac-dagger-map-weak} to find  $y'\in \overline{U}\SM U$ such that $f^{-1}(y')\cap \overline{W}\not=\emptyset$. Therefore $\overline{W}\SM W\not=\emptyset$.
\smallskip

Since $W$ is not closed and $X$ is weakly open-compact attainable (resp., $\kappa$-sequential), there is a point $z\in\overline{W}\SM W$ and a compact subset $K_1$ of $X$ such that $z\in K_1\cap \overline{K_1\cap W}$  (resp., a sequence $K_1=\{x_n\}_{n\in\w}\subseteq W$ converging to $z$). Set $y:=f(z)$ and $K:=f(K_1)$. Then $y\in \overline{U}\SM U$ and $K$ are such that $y\in K\cap \overline{K\cap U}$ (resp., $f(x_n)\to y$). Thus $Y$ is a weakly open-compact attainable space (resp., a $\kappa$-sequential space).\qed
\end{proof}

Since quotient maps are open, Lemma \ref{l:exa-dagger-map} and Propositions \ref{p:quot-oca} and \ref{p:quotient-woca} imply the next corollary.
\begin{corollary} \label{c:quotient-woca-group}
Let $H$ be a quotient group of a topological group $G$. If $G$ is open-compact attainable {\rm(}weakly open-compact attainable,  $\kappa$-sequential or $\kappa$-Fr\'{e}chet--Urysohn{\rm)}, then so is $H$.
\end{corollary}

\section{Characterization of (weakly) open-compact attainable and $\kappa$-sequential spaces} \label{sec:character-oca}

%%%%%%%%%%%%%%%%%%%%%%%
%%%%%%%%%%%%%%%%%%%%%%%
%%%%%%%%%%%%%%%%%%%%%%%
%%%%%%%%%%%%%%%%%%%%%%%
%%%%%%%%%%%%%%%%%%%%%%%

%\begin{definition} \label{def:function-woca} {\em
%Let $X$ and $Y$ be spaces. A surjective continuous map $f:X\to Y$ is called
%\begin{itemize}
%\item a {\em \dag-map} if $f^{-1}(y)\cap \overline{f^{-1}(U)}\not=\emptyset$ for every open non-closed set $U\subseteq Y$ and each $y\in \overline{U}\SM U$;
%\item a {\em weakly \dag-map} if  for every open non-closed set $U\subseteq Y$ there is  $y\in \overline{U}\SM U$ such that $f^{-1}(y)\cap \overline{f^{-1}(U)}\not=\emptyset$.
%\end{itemize} }
%\end{definition}

In this section we characterize (weakly) open-compact attainable spaces, $\kappa$-sequential spaces and $\kappa$-Fr\'{e}chet--Urysohn spaces using  $\kappa$-pseudo open maps and weakly $\kappa$-pseudo open maps introduced in Definition \ref{def:func-woca}.

We denote by $\KK(X)$ the family of all compact subsets of a space $X$.

\begin{theorem} \label{t:charact-oca}
For a space $X$, the following assertions are equivalent:
\begin{enumerate}
\item[{\rm(i)}] $X$ is an open-compact  attainable space;
\item[{\rm(ii)}] the natural inclusion map $I: \bigoplus\KK(X)\to X$ is a  $\kappa$-pseudo open  map;
\item[{\rm(iii)}] $X$ is a  $\kappa$-pseudo open map image of some locally compact space.
\end{enumerate}
\end{theorem}

\begin{proof}
(i)$\Ra$(ii) Assume that $X$ is open-compact  attainable. Consider the identity mapping
\[
I:\bigoplus_{K\in \KK(X)} K \to X
\]
We show that $I$ is a  $\kappa$-pseudo open map. Let $U\subseteq X$ be an open non-closed set, and let  $x\in \overline{U}\SM U$. Since $X$ is open-compact  attainable, there is $K\in\KK(X)$ such that $x\in K\cap \overline{U\cap K}$. Set $z:=K\cap I^{-1}(x)$. It remains to prove that $z\in I^{-1}(x)\cap \overline{I^{-1}(U)}$.
To this end, we note that $I{\restriction}_K$ is the identity map and hence
\[
\overline{I^{-1}(U)}\supseteq K\cap\overline{I^{-1}(U\cap K)}=K\cap I^{-1} \big( \overline{U\cap K}\big)=K\cap I^{-1} \big( K\cap\overline{U\cap K}\big).
\]
Therefore $z\in \overline{I^{-1}(U)}$ and hence $z\in I^{-1}(x)\cap \overline{I^{-1}(U)}$. Thus $I$ is  a \dag-map.
\smallskip

The implications (ii)$\Ra$(iii) is trivial. %follow from Proposition \ref{p:lc-oca}.
\smallskip

(iii)$\Ra$(i) follows from (i) of Proposition \ref{p:quot-oca} and (iii) of Proposition \ref{p:wOCA-k-space}. \qed %\ref{p:lc-oca}. \qed
\end{proof}

%In general one cannot omit the condition of being open-compact  attainable in (iv) and (v) of Theorem \ref{t:charact-oca}. Indeed, let $X=\varphi$ and $f:\varphi\to \varphi$ be the identity map. Then, by  Lemma \ref{l:exa-dagger-map}, $f$ is a \dag-map. However, $\varphi$ is not open-compact attainable by Example \ref{exa:OCA-varphi}.

Similar to Theorem \ref{t:charact-oca}, below we characterize weakly open-compact  attainable spaces.
\begin{theorem} \label{t:character-woca}
For a space $X$, the following assertions are equivalent:
\begin{enumerate}
\item[{\rm(i)}] $X$ is a weakly open-compact  attainable space;
\item[{\rm(ii)}] the natural inclusion map $I: \bigoplus\KK(X)\to X$ is a weakly  $\kappa$-pseudo open map;
\item[{\rm(iii)}] $X$ is a weakly  $\kappa$-pseudo open map image of some locally compact space;
\item[{\rm(iv)}] $X$ is a  weakly  $\kappa$-pseudo open map image of some  $k$-space;
\item[{\rm(v)}] $X$ is a  weakly   $\kappa$-pseudo open map image of some $k_\IR$-space.
\end{enumerate}
\end{theorem}

\begin{proof}
(i)$\Ra$(ii) Assume that $X$ is weakly open-compact  attainable. Consider the identity mapping
\[
I:\bigoplus_{K\in \KK(X)} K \to X
\]
We show that $I$ is a weakly  $\kappa$-pseudo open map. Let $U\subseteq X$ be an open non-closed set.  Since $X$ is weakly open-compact  attainable, there are a point $x\in \overline{U}\SM U$ and  $K\in\KK(X)$ such that $x\in K\cap \overline{U\cap K}$. Set $z:=K\cap I^{-1}(x)$. It remains to prove that $z\in I^{-1}(x)\cap \overline{I^{-1}(U)}$. To this end, we note that $I{\restriction}_K$ is the identity map and hence
\[
\overline{I^{-1}(U)}\supseteq K\cap\overline{I^{-1}(U\cap K)}=K\cap I^{-1} \big( \overline{U\cap K}\big)=K\cap I^{-1} \big( K\cap\overline{U\cap K}\big).
\]
Therefore $z\in \overline{I^{-1}(U)}$ and hence $z\in I^{-1}(x)\cap \overline{I^{-1}(U)}$. Thus $I$ is a weakly  $\kappa$-pseudo open map.
\smallskip

The implications (ii)$\Ra$(iii)$\Ra$(iv)$\Ra$(v) are trivial.
\smallskip

(v)$\Ra$(i) follows from (i) of Proposition \ref{p:quotient-woca} and  (i) of Proposition \ref{p:wOCA-k-space}.\qed
\end{proof}

Now we characterize $\kappa$-sequential spaces. For a space $X$, we denote by $\mathcal{S}(X)$ the family of all convergent sequences with limit point.

\begin{theorem} \label{t:character-kappa-seq}
For a space $X$, the following assertions are equivalent:
\begin{enumerate}
\item[{\rm(i)}] $X$ is a $\kappa$-sequential space;
\item[{\rm(ii)}] the natural inclusion map $I: \bigoplus\mathcal{S}(X)\to X$ is a weakly  $\kappa$-pseudo open map;
\item[{\rm(iii)}] $X$ is a weakly  $\kappa$-pseudo open map image of some metrizable locally compact space;
\item[{\rm(iv)}] $X$ is a  weakly  $\kappa$-pseudo open map image of some  sequential space;
\item[{\rm(v)}] $X$ is a  weakly   $\kappa$-pseudo open map image of some $s_\IR$-space.
\end{enumerate}
\end{theorem}

\begin{proof}
(i)$\Ra$(ii) Assume that $X$ is a $\kappa$-sequential space. Consider the identity mapping
\[
I:\bigoplus_{S\in \mathcal{S}(X)} S \to X
\]
We show that $I$ is a weakly \dag-map. Let $U\subseteq X$ be an open non-closed set.  Since $X$ is a $\kappa$-sequential space, there are a point $x\in \overline{U}\SM U$ and  $\{x_n\}_{n\in\w}\subseteq U$ such that $x_n\to x$. Set $S:=\{x_n\}_{n\in\w}\cup\{x\}$. Then $S\in \mathcal{S}(X)$ and $I{\restriction}_S$ is the identity map. Since $x\in I^{-1}(x)\cap \overline{I^{-1}(U)}$, it follows that  $I$ is a weakly  $\kappa$-pseudo open map.
\smallskip

The implications (ii)$\Ra$(iii)$\Ra$(iv)$\Ra$(v) are trivial.
\smallskip

(v)$\Ra$(i) follows from (ii) of Proposition \ref{p:quotient-woca} and (ii) of Proposition \ref{p:wOCA-k-space}.\qed
\end{proof}

Theorem 3.3 of \cite{LiL} states that a space $X$ is $\kappa$-Fr\'{e}chet--Urysohn if, and only if, $X$ is a $\kappa$-pseudo open image of a metric space. Below we give new characterizations of  $\kappa$-Fr\'{e}chet--Urysohn spaces.

\begin{theorem} \label{t:character-kappa-FU}
For a space $X$, the following assertions are equivalent:
\begin{enumerate}
\item[{\rm(i)}] $X$ is a $\kappa$-Fr\'{e}chet--Urysohn space;
\item[{\rm(ii)}] the natural inclusion map $I: \bigoplus\mathcal{S}(X)\to X$ is a  $\kappa$-pseudo open map;
\item[{\rm(iii)}] $X$ is a  $\kappa$-pseudo open map image of some metrizable locally compact space;
\item[{\rm(iv)}] $X$ is a  $\kappa$-pseudo open map image of some Fr\'{e}chet--Urysohn space.
\end{enumerate}
\end{theorem}

\begin{proof}
(i)$\Ra$(ii) Assume that $X$ is a $\kappa$-Fr\'{e}chet--Urysohn space. Consider the identity mapping
\[
I:\bigoplus_{S\in \mathcal{S}(X)} S \to X
\]
We show that $I$ is a  $\kappa$-pseudo open map. Let $U\subseteq X$ be an open non-closed set and let $x\in \overline{U}\SM U$ be an arbitrary point.  Since $X$ is a $\kappa$-Fr\'{e}chet--Urysohn space, there is a sequence  $\{x_n\}_{n\in\w}\subseteq U$ converging to $x$. Set $S:=\{x_n\}_{n\in\w}\cup\{x\}$. Then $S\in \mathcal{S}(X)$ and $I{\restriction}_S$ is the identity map. Since $x\in I^{-1}(x)\cap \overline{I^{-1}(U)}$, it follows that  $I$ is a \dag-map.
\smallskip

The implications (ii)$\Ra$(iii)$\Ra$(iv) are trivial.
\smallskip

(iv)$\Ra$(i) follows from (ii) of Proposition \ref{p:quot-oca}.\qed % and (ii) of Proposition \ref{p:wOCA-k-space}.\qed
\end{proof}

%The space $X=\varphi$ is sequential  but not $\kappa$-Fr\'{e}chet--Urysohn (Example \ref{exa:OCA-varphi}),  and  the identity map $f:\varphi\to \varphi$ is a \dag-map (Lemma \ref{l:exa-dagger-map}). Therefore, in general, one cannot omit the condition of being $\kappa$-Fr\'{e}chet--Urysohn in (iv) and (v) of Theorem \ref{t:character-kappa-seq}.

%%%%%%%%%%%%%%%%%%
%%%%%%%%%%%%%%%%%%
%%%%%%%%%%%%%%%%%%
%%%%%%%%%%%%%%%%%%
%%%%%%%%%%%%%%%%%%

\section{Locally compact groups and locally compact abelian groups with the Bohr topology} \label{sec:LCG-LCA-Bohr}

%%%%%%%%%%%%%%%%%%
%%%%%%%%%%%%%%%%%%
%%%%%%%%%%%%%%%%%%
%%%%%%%%%%%%%%%%%%
%%%%%%%%%%%%%%%%%%

We start this section with the following general theorem. Recall that a topological group $G$ is called {\em feathered} if it contains a nonempty compact set $K$ of countable character in $G$.

%A topological group $G$ is feathered if and only if $G$ is a (paracompact) $p$-space \cite[Theorem 4.3.35]{ArT}.

\begin{theorem} \label{t:feathered-group-kFU}
Each feathered group $G$ is $\kappa$-Fr\'{e}chet--Urysohn.
\end{theorem}

\begin{proof}
Let $W\subseteq G$ be open and $g\in \overline{W}\setminus W$. We have to prove that there exists a sequence $(g_n)_{n\in\w}\subseteq W$ converging to $g$.

Let $\mathcal{H}$ be a family of compact subgroups of $G$ of countable character in $G$. Set
\[
\mathcal{B} =\{ UH: U\subseteq G \mbox{ is open in } G \mbox{ and } H\in\mathcal{H}\}.
\]
By definition, if $U\in \mathcal{B}$, then there is $H_U\in\mathcal{H}$ such that $U=UH_U$.

{\em Claim 1. The family $\mathcal{B}$ satisfies the following conditions:
\begin{enumerate}
\item[{\rm (i)}] if $x\in G$ and $U\in\mathcal{H}$, then $xU\in\mathcal{B}$;
\item[{\rm (ii)}] $\mathcal{B}$ is a base for $G$;
\item[{\rm (iii)}] if $U,V\in \mathcal{B}$, then $U\cap V\in \mathcal{B}$;
\item[{\rm (iv)}] if $(U_n)_{n\in\w}\subseteq\mathcal{B}$, then $\bigcup_{n\in\w}U_n\in \mathcal{B}$.
\end{enumerate}
}

Indeed, if $U=UH$ for some $H\in \mathcal{H}$, then $xU=(xU)H\in\mathcal{B}$ that proves (i).

(ii) is Corollary 4.3.12 of \cite{ArT}. Let us give a simple direct proof.  By (i), it suffices to verify that the family $\mathcal{B}_0=\{U\in\mathcal{B}:e\in U\}$ is a base at the identity $e$. Let $S$ be a neighborhood of $e$. Then $V^2\subseteq S$ for some neighborhood $V$ of $e$. Proposition 4.3.11 of \cite{ArT}, there exists $H\in \mathcal{H}$ such that $H\subseteq V$. Then $VH\in\mathcal{B}_0$ and $VH\subseteq V^2\subseteq S$.

To prove (iii), choose  $H_U,H_V\in\mathcal{H}$ such that $UH_U=U$ and $VH_V=V$. Set $H=H_U\cap H_V$. Then $H\in\mathcal{H}$ and $(U\cap V)H=U\cap V$.

For every $n\in\w$, take $H_n\in\mathcal{H}$ such that $U_nH_n=U_n$. Set $H=\bigcap_{i\in\w}H_n$. Then $U_nH=U_n$ for every $n\in\w$. Therefore, $\big(\bigcup_{n\in\w}U_n\big)H=\bigcup_{n\in\w}U_n$. This proves (iv), and hence Claim 1 is proved.
\smallskip

The next claim means that we can replace $W$ by some $U\in\mathcal{B}$.
\smallskip

{\em Claim 2. There exists $U\in\mathcal{B}$ such that $g\in \overline{U}$ and $U\subseteq W$.}

Indeed, set $\gamma:=\{V\in\mathcal{B}:V\subseteq W\}$. Since, by (ii) of Claim 1,  $\mathcal{B}$ is a base of $G$, we have $\bigcup \gamma=W$.
Since $G$ is feathered, the $o$-tightness of $G$ is countable \cite[Corollary 5.5.7]{ArT}. Therefore, $g\in \overline{\bigcup\mu}$ for some at most countable $\mu\subseteq\gamma$. Let $U=\bigcup\mu$. Then, by (iv) of Claim 1, we have $U\in \mathcal{B}$. Since  $g\in \overline{U}$ and $U\subseteq W$, the set $U$ is as desired. Claim 2 is proved.
\smallskip

Fix $U\in\mathcal{B}$ which satisfies Claim 2. Take $H\in\mathcal{H}$ such that $UH=U$. Let $G/H=\{xH:x\in G\}$ be a left quotient space.
Since $H$ is a compact subgroup, the quotient mapping
\[
\pi: G\to G/H,\ x\mapsto xH
\]
is perfect \cite[Theorem 1.5.7]{ArT} and open \cite[Theorem 1.5.1]{ArT}. Set $x:=\pi(g)$. Since $g\in \overline{U}$, we have $x\in \overline{\pi(U)}$. As $H$ is a compact subgroup of countable character in $G$, $G/H$ is metrizable \cite[Lemma 4.3.19]{ArT}. Therefore, there exists a sequence $(x_n)_{n\in\w}\subseteq \pi(U)$ converging to the point $x=\pi(g)$. Put $S:=\{x_n:n\in\w\}$ and $Q=\{x\}\cup S$. The set $Q$ is compact in the metrizable space $G/H$, and hence $Q$ is a set of type $G_\delta$ in $G/H$. Since $\pi$ is a surjective continuous perfect mapping, the set $K=\pi^{-1}(S)$ is a compact set of type $G_\delta$ in $G$. Then, by Corollary 10.3.9 of  \cite{ArT},  $K$ is a Dugundji compact space. Consequently, by Theorem 10.1.3 of  \cite{ArT}, $K$ is a dyadic compact space.
Therefore, by Corollary 3.5 of \cite{LiL}, $K$ is a $\kappa$-Fr\'{e}chet--Urysohn space. Since $UH=U$, we have $U=\pi^{-1}(\pi(U))$. Consequently, $K\cap U=\pi^{-1}(S)$. Since $\pi$ is open, we have $gH=\pi^{-1}(x)\subseteq \overline{K\cap U}$. Consequently, $g\in \overline{K\cap U}$. As $K$ is a $\kappa$-Fr\'{e}chet--Urysohn space, there exists a sequence $(g_n)_{n\in\w} \subseteq K\cap U \subseteq W$ converging to $g$.\qed
\end{proof}

Since all \v{C}ech-complete groups and hence all locally compact groups are feathered \cite[4.3]{ArT}, we obtain the following generalization of Proposition \ref{p:compact-group-kFU}.

\begin{corollary} \label{c:lcg-kFU}
Each \v{C}ech-complete group is a $\kappa$-Fr\'{e}chet--Urysohn space. Consequently, every locally compact group is a $\kappa$-Fr\'{e}chet--Urysohn space.
\end{corollary}

%\begin{problem}
%Characterize compact (locally compact or \v{C}ech-complete) spaces which are $\kappa$-sequential.
%\end{problem}

Let $G$ be a locally compact abelian group. The group $G$ endowed with the Borh topology $\tau_b$ induced from the Bohr compactification $bG$ of $G$ is denoted by $G^+$. Numerous topological and algebraic-topological properties of the precompact group $G^+$ are well-studied, we refer the reader to Chapter 9 of \cite{ArT} and references therein. The next result lies in this line of researches.

\begin{theorem} \label{t:LCA-Bohr-kFU}
For a locally compact abelian group $G$, the following assertions are equivalent:
\begin{enumerate}
\item[{\rm(i)}] $G^+$ is a $\kappa$-Fr\'{e}chet--Urysohn space;
\item[{\rm(ii)}] $G^+$ is an open-compact attainable space;
\item[{\rm(iii)}] $G^+$ is an Ascoli space;
\item[{\rm(iv)}] $G^+$ is a sequentially Ascoli space;
\item[{\rm(v)}]  $G$ is a compact space;
\item[{\rm(vi)}] $G^+$ is a $k'$-space.
\end{enumerate}
\end{theorem}

\begin{proof}
The implications (i)$\Ra$(ii)$\Ra$(iii)$\Ra$(iv) and (v)$\Ra$(vi)$\Ra$(ii) are clear. The implication (iv)$\Ra$(v) is proved in Theorem 3.1 of \cite{Gabr-seq-Ascoli}, and the implication (v)$\Ra$(i) follows from Corollary \ref{c:lcg-kFU} and the fact that $G^+=G$ for any compact abelian group.\qed
\end{proof}

\begin{corollary} \label{c:LCA-Bohr-FU}
Let $G$ be a locally compact abelian group. Then  $G^+$ is a Fr\'{e}chet--Urysohn space if, and only if, $G$ is a compact Fr\'{e}chet--Urysohn space.
\end{corollary}

To consider the property of being $\kappa$-sequential or weakly open-compact attainable for locally compact abelian groups with the Bohr topology we use the following assertion.

\begin{theorem} \label{t:prod-kappa-connected}
Let $X$ and $Y$ be $\kappa$-sequential {\rm(}resp., weakly open-compact attainable{\rm)} spaces. If $X$ is connected, then $X\times Y$ is a $\kappa$-sequential {\rm(}resp., weakly open-compact attainable{\rm)} space.
\end{theorem}

\begin{proof}
Suppose for a contradiction that there is an open non-closed set $W\subseteq X\times Y$ such that for every $z\in \overline{W}\SM W$ there is no a sequence (resp., a relatively compact in  $X\times Y$ subset) $A\subseteq W$ which converges to $z$ (resp.,  $z\in \overline{A}$).

We claim that there is an open $U\subseteq Y$ such that $W=X\times U$. Indeed, otherwise, there would exist a point $(x',y_\ast)\in W$ such that the open nonempty subset $V:=\{x\in X: (x,y_\ast)\in W\}$ of $X$ is not equal to $X$. We show that $V$ is also closed. To this end, assuming for a contradiction and taking into account that $X$ is $\kappa$-sequential (resp., weakly open-compact attainable), we could find a point $x_\ast\in \overline{V}\SM V$ and a sequence (resp., a relatively compact in $X$ subset) $S\subseteq V$ such that $x_\ast\in \overline{S}$. Then the point $z=(x_\ast,y_\ast)$ belongs to $\overline{W}\SM W$ and the sequence (resp., the relatively compact in  $X\times Y$ subset) $A:=(S,y_\ast)\subseteq W$ converges to $z$ (resp.,  $z\in \overline{A}$). But this contradicts our supposition on $W$. Therefore $V$ must be also closed. Since $X$ is connected, it follows that the clopen nonempty subset $V$ of $X$ must coincide with $X$. This contradiction shows that $W=X\times U$ for some open subset $U$ of $Y$.

Since $W$ is not closed, we obtain that also $U$ is not closed. Since $Y$ is $\kappa$-sequential (resp., weakly open-compact attainable), there are a point $y_0\in \overline{U}\SM U$ and a sequence (resp., a relatively compact in $Y$ subset) $S_0\subseteq U$ such that $y_0\in \overline{S_0}$. Then for every $x\in X$, the point $z_0=(x,y_0)$ belongs to $\overline{W}\SM W$ and the sequence (resp., the relatively compact subset) $A_0:=(x,S_0)\subseteq W$ converges to $z_0$ (resp.,  $z_0\in \overline{A_0}$). However this contradicts our assumption on $W$.\qed
\end{proof}

Corollary 3.5 of \cite{LiL} states that any dyadic compact space is $\kappa$-Fr\'{e}chet--Urysohn. This result and Theorem \ref{t:prod-kappa-connected} immediately imply the following corollary.
\begin{corollary} \label{c:k-sequential-dyadic}
If $X$ is a connected $\kappa$-sequential {\rm(}resp., weakly open-compact attainable{\rm)} space and $K$ is a dyadic compactum, then the product $X\times K$ is a $\kappa$-sequential {\rm(}resp., weakly open-compact attainable{\rm)} space.
\end{corollary}

By Corollary \ref{c:lcs-k-sequential}, each topological vector space, being even pass-connected, is $\kappa$-sequential. Now Corollary \ref{c:k-sequential-dyadic} implies the next assertion.
\begin{corollary} \label{c:tvs-dyadic-product}
If $E$ is a topological vector space and $K$ is a dyadic compactum, then $E\times K$ is a $\kappa$-sequential space.
\end{corollary}

We finish this section with the following theorem.

\begin{theorem} \label{t:LCA-Bohr-woca}
For a locally compact abelian group $G$, the following assertions are equivalent:
\begin{enumerate}
\item[{\rm(i)}] $G^+$ is a $\kappa$-sequential space;
\item[{\rm(ii)}] $G^+$ is a weakly open-compact attainable space;
\item[{\rm(iii)}] $G$ is topologically isomorphic to $\IR^n\times H$, where $n\in\w$ and $H$ is a compact group.
\end{enumerate}
\end{theorem}

\begin{proof}
(i)$\Ra$(ii) is clear.
\smallskip

(ii)$\Ra$(iii) Assume that $G^+$ is weakly open-compact attainable. By the structural Theorem 24.30 of \cite{HR1}, $G$ is topologically isomorphic to  $\IR^n\times H$, where $n\in\w$ and $H$ has  an open compact subgroup $K$. Denote by $S$ the discrete group $G/(\IR^n\times K)$.  By Problem 9.9.N of \cite{ArT}, we have $S^+=G^+/(\IR^n\times K)^+$. Hence, by Corollary \ref{c:quotient-woca-group}, the precompact group $S^+$ is weakly open-compact attainable. By Theorem 9.9.30 of \cite{ArT}, $S^+$ has no infinite compact sets. The last two facts imply that $S$ must be finite. Therefore $H$ is a compact group.
\smallskip

(iii)$\Ra$(i) Let $G$ be topologically isomorphic to  $\IR^n\times H$, where $n\in\w$ and $H$ is a compact group. Then, by Problem 9.9.N of \cite{ArT}, $G^+=(\IR^n)^+\times H$. By Corollary  \ref{c:lcg-kFU},  $H$ is a $\kappa$-sequential space. It is clear that $(\IR^n)^+$ is pass-connected, and hence, by Proposition \ref{p:pass-connected-k-sequential},  $(\IR^n)^+$ is $\kappa$-sequential. Thus, by Theorem  \ref{t:prod-kappa-connected}, $G^+$ is a $\kappa$-sequential space.\qed
%=\bigcup_{k\in\w} X_k$, where $X_k:=[-k,k]^n$ for every $k\in\w$.  By Exercise 4.5.9(b) of \cite{Eng}, $X_k$ is dyadic. Now Proposition \ref{p:prod-ark-compact} implies that $G^+$ is a $\kappa$-sequential space.\qed
\end{proof}

\bibliographystyle{amsplain}

\end{document}